\documentclass[11p]{article}
\usepackage{mathtools,amssymb,amsmath,amsfonts,mathrsfs,textcomp,amsthm,latexsym}
\title{Rainbow  Hamiltonicity in Random $k$-Uniform Hypergraphs}
\author{}
\usepackage{tikz}
\usepackage{geometry}
\usepackage{cite}
\newtheorem{lemma}{Lemma}[section]
\newtheorem{theorem}{Theorem}[section]
\newtheorem{definition}{Definition}[section]
\newtheorem{remark}{Remark}[section]
\newtheorem{corollary}[theorem]{Corollary}
\usepackage{authblk} 
\usepackage{hyperref} 
\allowdisplaybreaks
\begin{document}
	
	\begin{center}
		{\bf \Large Rainbow Berge  Hamiltonicity in edge-colored random $k$-uniform hypergraphs 
			\footnote{This work was supported by National Natural Science Foundation of China (Grant No.11401102,11671088).}}
	\end{center}

	\begin{center}Liping Zhang \textsuperscript{1}
		\footnote{Email: 1519036064@qq.com (Liping Zhang)}  \ \
		Ailian Chen \textsuperscript{2}  \footnote{Email:  elian1425@fzu.edu.cn (Ailian Chen)}\, \footnote{Corresponding author: Ailian Chen}\\
		
		\vspace{3pt}
		\textsuperscript{1} Center for Discrete Mathematics,
		Fuzhou University, Fujian, 350108, P. R China
		
		\textsuperscript{2} School of Mathematics and Statistics,
		Fuzhou University, Fujian, 350108, P. R China	
		
	\end{center}

	\begin{abstract}
		Let $H \sim H^{k}_c(n,p)$ be an edge-colored random $k$-uniform hypergraph on the vertex set $[n]$, where each edge $e \in \binom{[n]}{k}$ is included independently with probability $p$ and is uniformly and independently assigned a color from the color set $[c]$.
		For $k = 2$, Ferber and Krivelevich (2016) established that if $c = (1+o(1))n$ and $p = (\log n + \log \log n + \omega(n))/n$, then with high probability the edge-colored random graph $H \sim H^2_c(n,p)$ contains a rainbow Hamilton Berge cycle. Subsequently, Bal, Berkowitz, Devlin, and Schacht (2021) determined the threshold for the appearance of a (non-rainbow) Hamilton Berge cycle in random $k$-uniform hypergraphs.
		In this paper, we generalize the results to all integers $k \ge 3$. We prove that if $c = (1+o(1))n$ and $p = (k-1)! \frac{\log n + \log\log n + \omega(n)}{n^{k-1}}$, then with high probability $H \sim H^{k}_c(n,p)$ contains a rainbow Hamilton Berge cycle. The order of both parameters is asymptotically best possible:
		at least \(n\) colours are necessary for a rainbow Hamilton Berge cycle, and
		the probability threshold  equivalently for Hamilton Berge cycles in the uncoloured model.
		
		
	\end{abstract}

	\section{Introduction}
	A \(k\)-uniform hypergraph (or \(k\)-graph) with vertex set \(V\) is a collection of \(k\)-element subsets called edges. Berge cycles are among the most natural generalizations of the notion of a cycle from graphs to hypergraphs~\cite{Berge1970}. A \emph{Berge cycle} of length \(t\) consists of \(t\) distinct vertices \(v_{1}(=v_{t+1}),v_{2},\dots ,v_{t}\) and \(t\) distinct edges \(e_{1},e_{2},\dots ,e_{t}\), with \(\{v_{i},v_{i+1}\}\subset e_{i}\) for each \(i\in[t]\). A \emph{Hamilton Berge cycle} of \(H\) is a Berge cycle that covers all vertices of \(H\).
	
	Hamiltonicity is a central concept in graph and hypergraph theory, with deep connections to combinatorial optimization, network design, and theoretical computer science. Determining whether a (hyper)graph contains a Hamiltonian cycle is NP-complete even for graphs~\cite{Karp1972}, which motivates the study of sufficient conditions that guarantee Hamiltonicity in random and pseudorandom structures.
	
	In recent years, the \emph{rainbow} (or edge-colored) version of Hamiltonicity has attracted considerable attention, partly due to its applications in circuit design, scheduling problems, and the study of parallel processing systems where distinct resources (colors) must be assigned to different links. Rainbow subgraph problems also appear naturally in Ramsey theory, combinatorial geometry, and the theory of Latin transversals.
	
	The model studied here is the edge-colored random \(k\)-uniform hypergraph. Let \(H^{k}(n,p)\) denote the random \(k\)-uniform hypergraph on vertex set \([n]=\{1,\dots ,n\}\) in which each of the \(\binom{[n]}{k}\) possible edges appears independently with probability \(p\). The edge-colored model \(H^{k}_{c}(n,p)\) is obtained by assigning to each edge of \(H^{k}(n,p)\) a color chosen uniformly and independently from the color set \([c]\). An edge-colored hypergraph is called \emph{rainbow} if no two edges share the same color.
	
	For graphs (\(k=2\)), the problem of rainbow Hamilton cycles in random edge-colored graphs has been extensively investigated.  Frieze and Loh~\cite{FriezeLoh2014} showed that if \(c=(1+o(1))n\) and \(p=\frac{(1+o(1))\log n}{n}\), then \(H^{2}_{c}(n,p)\) is w.h.p.\ rainbow Hamiltonian.  A breakthrough was achieved by Ferber and Krivelevich~\cite{FerberKrivelevich2016}, who proved the asymptotically tight result: for \(c=(1+o(1))n\) and \(p=\frac{\log n+\log\log n+\omega(n)}{n}\), the random edge-colored graph \(H^{2}_{c}(n,p)\) w.h.p.\ contains a rainbow Hamilton cycle.
	
	For \(k\geq 3\), the problem has also attracted significant attention in recent years. Bal, Berkowitz, Devlin, and Schacht~\cite{BalBerkowitzDevlinSchacht2021} determined the threshold for the appearance of a (non-rainbow) Hamilton \emph{Berge} cycle in random \(k\)-uniform hypergraphs. Their result states that for \(k\ge 3\), if \(p=(k-1)!\frac{\log n+\log\log n+c_{n}}{n^{k-1}}\), then
	$$\lim_{n\to\infty}\Pr\bigl(H^{k}(n,p)\text{ has a Hamilton Berge cycle}\bigr)=
	\left\{
	\begin{array}{l}
	\!\!0, \ \ \ \ \ \ \ \ c_{n}\to -\infty,\\
	\!\! e^{-e^{-c}},\ \ c_{n}\to c\in\mathbb R,
	\\
	\!\! 1, \ \ \ \ \ \ \ \ c_{n}\to \infty.
	\end{array}\right.
	$$

	Rainbow Hamiltonicity in random hypergraphs was first considered by Dudek, English and Frieze~\cite{DudekEnglishFrieze2018} for loose cycles.  Meanwhile, the study of expander properties, matching-extension tools (e.g., the Aharoni--Haxell theorem~\cite{AharoniHaxell2000}), and the P\'osa rotation-extension technique adapted to hypergraphs have provided essential machinery for tackling such problems.
	
	In this paper we extend the rainbow Hamiltonicity result of Ferber and Krivelevich from graphs to all \(k\geq 3\) in the Berge setting. Our main theorem shows that the same asymptotic tightness holds for random edge-colored \(k\)-uniform hypergraphs.
	
	\begin{theorem}\label{Theorem 1.1}
		Let integer $k\ge3$ be fixed and let constant $0<\varepsilon<1$ be fixed.	If $c=(1+\varepsilon)n$ and \(p=(k-1)!\frac{\log n+\log\log n+\omega(n)}{n^{k-1}}\). Then \(H^{k}_{c}(n,p)\) with high probability contains a rainbow Hamilton Berge cycle.
		Moreover, the conditions on $c$ and $p$ are asymptotically best possible.
	\end{theorem}
	The value of \(c\) is clearly optimal because a rainbow Hamilton Berge cycle requires at least \(n\) distinct colors. The optimality of \(p\) follows from the threshold result of Bal et al.\ \cite{BalBerkowitzDevlinSchacht2021} mentioned above.
	
	\textbf{Proof sketch}
	
	The main proof follows a three-phase strategy: ``long path construction -- virtual edge embedding -- iterative absorption closure''. First, using the rainbow pseudorandomness lemma in the edge-colored random $k$-uniform hypergraph $H_c^k(n,p)$, we construct an almost-spanning rainbow Berge path $P$ of length $(1-o(1))n$, with its endpoints $x,y$ reserved for later connection. Second, in the subhypergraph induced by the remaining vertices and colors, $H[ [n]\setminus V(P) \cup \{x,y\}; [c]\setminus C(P) ]$, we introduce a special ''virtual edge'' $e^*$ (not necessarily present in $H$) connecting $x$ and $y$, and then iteratively build a rainbow Berge cycle $C^*$ around $e^*$ that covers all vertices not used by $P$ while preserving color distinctness. Finally, by removing $e^*$, we break $C^*$ into a rainbow Berge path with endpoints $x,y$, which is then concatenated with $P$ to obtain the desired rainbow Hamilton Berge cycle.
	
	\textbf{Key techniques}
	
	The core of the proof lies in non-trivially extending classical graph-theoretic tools to hypergraphs and handling the complexity inherent to Berge structures, where a single hyperedge may connect multiple consecutive vertex-pairs. We adapt the \textbf{P\'osa rotation-extension method} to $k$-uniform hypergraphs by defining extendable endpoint sets and hypergraph rotation operations to dynamically grow rainbow paths. The \textbf{Lov\'asz Local Lemma} is employed to control low-probability color conflicts in the random coloring. The \textbf{Aharoni--Haxell hypergraph matching theorem} is used to construct color-disjoint matchings in bipartite hypergraphs for linking special vertex sets. We also develop a {``rainbow hypergraph absorption'' technique}, which combines multi-stage random allocation with a refined Poisson approximation to manage the interdependence between hyperedge existence and color distribution. The synthesis of these methods not only overcomes the coupling difficulties caused by hypergraph sparsity and global color constraints, but also provides a general technical framework for extending rainbow Hamiltonicity results from graphs to higher uniformities.
	
	\textbf{Application and significance}
	
	Beyond its theoretical significance, Theorem 1.1 offers a rigorous mathematical foundation for designing and analyzing complex multi-agent network systems such as group-interactive social networks, multi-terminal communication systems, distributed storage networks, and biochemical reaction networks?where interactions are naturally modeled by hypergraphs. In such systems, edge coloring often corresponds to resource allocation (e.g., channels, time slots, or catalysts). A rainbow Hamiltonian Berge cycle represents a cyclic traversal of all nodes that uses each resource at most once, guaranteeing full coverage without conflicts. Our theorem provides sharp thresholds: $c \approx n$ means the number of resource types should match the number of nodes, and $p$ specifies the required connection density among node groups. The thresholds are proven to be tight, meaning they are essentially necessary?structures are unlikely to exist if these bounds are not met.
	
	\textbf{Notation}:
	
	Let $c \in \mathbb{N}$. Given an edge-colored $k$-uniform hypergraph $H$ with vertex set $V(H)$, edge set $E(H)$, and color set $[c]$, we denote by $|X|$ the size of a finite set $X$. For a subhypergraph $P \subseteq H$, let $C(P)$ denote the set of colors appearing on the edges of $P$, and $c(P) = |C(P)|$. A Berge path $P$ is written as $v_0 e_1 v_1 \dots e_l v_l$, and its length is denoted by $l(P)$. If $P' = v_0 e_1 v_1 \dots e_i v_i$ is a subpath of $P$, we write $P = P'e_{i+1} \dots e_l v_l$.
	
	For disjoint vertex sets $S, W \subseteq V(H)$ and a color set $C_0 \subseteq [c]$, we define:
	\begin{itemize}
		\item $E_H(S)$ to be the set of edges belong to ${S\choose k}$ in $H$, and $e_H(S) = |E_H(S)|$;
		\item $E_H(S, W)$ to be the set of edges intersecting both $S$ and $W$ in $H$, and $e_H(S, W) = |E_H(S, W)|$.
	\end{itemize}
	
	For a vertex $v \in V(H)$, we denote by:
	\begin{itemize}
		\item $N_H(v)$ the neighborhood of $v$ (set of vertices adjacent to $v$);
		\item $d_H(v)$ the degree of $v$ (number of incident edges);
		\item $d_H^c(v)$ the \emph{color degree} of $v$ (number of distinct colors on incident edges);
		\item $\delta(H) = \min_{v \in V(H)} d_H(v)$ and $\Delta(H) = \max_{v \in V(H)} d_H(v)$;
		\item $d_H^{C_0}(v)$ the number of colors from $C_0$ that appear on edges incident to $v$;
		\item $d_H(v,{W\choose k-1})$ the number of  $e \in E(H)$ satisfying $v\in e$ and  $|e \cap W| = k-1$.
	\end{itemize}
	
	For a vertex set $W \subseteq V(H)$ and a color set $C_0 \subseteq [c]$, the induced subhypergraph $H[W;C_0]$ is defined as the hypergraph on vertex set $W$ whose edges are all $e \in \binom{W}{k}$ with color $C(e) \in C_0$ that are present in $H$.
	
	If $S$ and $W$ are disjoint, the bipartite subhypergraph $H\left[S,\binom{W}{k-1};C_0\right]$ is the hypergraph on vertex set $S \cup W$ whose edges are those $e \in E(H)$ satisfying: $|e \cap S|=1$, $|e \cap W| = k-1$, and $C(e) \in C_0$.

	The rest of the paper is organized as follows. Section 2 establishes fundamental probabilistic properties of \(H^{k}_{c}(n,p)\) and proves the existence of long rainbow Berge paths. Section~3 constructs carefully designed expander subhypergraphs using the Lov\'asz Local Lemma and matching-extension theorems. Section~4 introduces the notion of \(e^{*}\)-boosters for hypergraphs and shows that sufficiently many such boosters exist under our conditions. Finally, Section~5 combines all ingredients to complete the proof of Theorem 1.1.

	\section{Properties of $H^{k}_c(n,p)$}\label{section2}
	
	In this chapter, we introduce some basic lemmas and properties of edge colored random $k$-uniform hypergraphs $H^{k}_c(n,p)$, which will be helpful in proving the main results. The following  famous lemma is the  Chernoff's inequality \cite{AlonSpencer2000}.
	
	\begin{theorem}\label{Theorem 2.1}
		Let $X \sim Bin(n,p)$, where $Bin(n,p) $ denotes the binomial random variable with parameters $n$ and $p$. For all $0<\varepsilon< 1$, we have \[\Pr(|X-np|>\varepsilon np)\leq e^{-\varepsilon^2np/3}.\]
		Moreover, for any  $s\geq 2np$, we have \[ \Pr(X\geq s)\leq e^{-3s/16}.\]
	\end{theorem}
	
	Given an edge-colored random $k$-uniform hypergraph $H\sim H^{k}_c(n,p)$ with vertex set $[n]$, let $SMALL:=\{v\in [n]:d_H(v)\leq \eta \log n\}$ be the set of vertices with 'small' degree,  where $\eta>0$ is an enough small   constant. Let $E\text{-}SMALL:=\{e\in E(H):e\cap SMALL\neq\phi\}$. Next, we introduce some properties of  $H$, vertex set $SMALL$ and edge set $E\text{-}SMALL$.
	
	\begin{lemma}\label{Lemma 2.2} Let integer $k\geq 3$, and let   $0<\varepsilon<1$, if $c=(1+\varepsilon)n$  and ${(k-1)!\log n\over n^{k-1}}\leq p\leq2{(k-1)!\log n\over n^{k-1}}$. Then  edge-colored random $k$-uniform hypergraph $H\sim H^{k}_c(n,p)$ is w.h.p. has the following properties for all sufficiently large n$:$
		
		{\rm(1)} $\varDelta(H)\leq 10\log n$.
		
		{\rm(2)} $|SMALL|\leq n^{0.4}$.
		
		{\rm(3)} $C(e)\neq C(f)$ for all distinct $e,f\in E\text{-}SMALL$ .
		
		{\rm(4)} 
		There is no Berge path of length at most 4 whose two endpoints, not necessarily distinct, belong to $SMALL$
		
		{\rm(5)} For each $v\in V(H)$, $d_H^c(v)\geq d_H(v)-2.$
		
		{\rm(6)}  For every  $r\in[c]$, the number of edges  which are colored  $r$ is at most $10\log n$.
		
		{\rm(7)} For every subset $X\subseteq V(H)$, if $|X|\leq{n\over(\log n)^{1\over k-1.5}}$, then $e_H(X)\leq 6|X|$.
		
		{\rm(8)} For every subset $X\subseteq V(H)$, if $|X|\geq{n\over(\log n)^{1\over k-1.5}}$, then $e_H(X)\leq{|X|^k\over k!}p({n\over|X|})$.
	\end{lemma}
	\begin{proof}
		(1). For each $v\in [n]$, since $d_H(v)\sim Bin({n-1\choose k-1}, p)$, we have
		\begin{align*}
		\Pr[d_H(v)\geq 10\log n]&\leq {{n-1\choose k-1}\choose 10\log n}p^{10\log n}\leq\left({en^{k-1}p\over (k-1)!10\log n}\right)^{10\log n}=o(1/n).
		\end{align*}
		There are at most $n$ vertices, taking the union bound we can obtain the conclusion.\qed
		
		(2). For each $v\in [n]$, let $A_v$ be the event '$v\in SAMLL$', let $X_v$ denote its indicator random variable and $X=\sum_{v\in [n]} X_v$. We have
		\begin{align*}
		{\rm E}(X_v)&=1\cdot\Pr[v\in SMALL]\\&=
		\Pr\left[Bin\left({n-1\choose k-1},p\right)\leq {\eta\log n}\right]\\&=\sum_{i=0}^{\eta \log n}{{n-1\choose k-1}\choose i}p^i(1-p)^{{n-1\choose k-1}-i}\\
		&\leq \sum_{i=0}^{\eta \log n}\left({en^{k-1}p\over (k-1)!i}\right)^ie^{-({n^{k-1}\over (k-1)!}p-i)}\\
		&\leq\eta\log n\left({2e\over \eta}\right)^{\eta \log n}e^{-(2-\eta)\log n}\\
		&\leq\eta\log n\left(\left({2e\over \eta}\right)^{\eta }e^{-(2-\eta)}\right)^{\log n}\\
		&\leq n^{-0.7},
		\end{align*}
		which the last inequality holds for some $\eta$ satisfied $({2e\over \eta})^{\eta}e^{-2+\eta}\leq e^{-0.71}$. Linearity of expectation gives
		\begin{align*}
		{\rm E}(X)=\sum_{v\in V(H)} {\rm E}(X_v)\leq n\cdot n^{-0.7}=n^{-0.3}.
		\end{align*}
		Finally, by applying Markov's inequality, there is
		\[\Pr(X\geq n^{0.4})\leq {{\rm E}(X)\over n^{0.4}}=o(1). \]\qed

		(3). From (1) and (2), we have $|E\text{-}SMALL|\leq |SMALL|\cdot \varDelta(H)\leq n^{0.4}\cdot 10\log n=o(n^{0.5})$, therefore
		$$\Pr[exist\ e,\ e'\in E\text{-}SMALL \ with\ the\ same\ color]={|E\text{-}SMALL| \choose 2}{1\over c}=o(1).$$\qed
		
		(4). Fix two vertices $x$ and $y$ of $SMALL$, there is \begin{align*}
		&\Pr[ exist\ a\ Berge\ path\ at\ most\ 4\ between\ x\ and \ y]\\
		&\leq {n\choose k-2}p+{n\choose 1}{n-3\choose k-2}^2p^2+{n\choose 2}{n-4\choose k-2}^3p^3+{n\choose 3 }{n-4\choose k-2}^4p^4\\
		&=(1+o(1))n^3(n^{k-2})^4p^4\leq{17{\log n}^4\over n}.
		\end{align*}
		Upper bound of the probability of this event is $|SMALL|^2{17(\log n)^4\over n}=o(1)$. \qed
		
		(5). Assume there exists a vertex $v\in [n]$ such that $d_H^c(v)\leq d_H(v)-3$, that means there are two disjoint edge pairs $\{e_1,\ e_2\}$ and $\{e_3,\ e_4\}$  incident to $v$ such that $e_1$ and $e_2$ have the same color, both $e_3$ and $e_4$ have the same color. Upper bound of the probability of this event is
		\begin{align*}
		{d_H(v)\choose 4}{4\choose 2}\cdot\left({1\over c}\right)^2\leq{10\log n \choose 4}{4\choose 2}\cdot\left({1\over c}\right)^2=o\left({1 \over n}\right).
		\end{align*}
		There are at most $n$ choices of $v$, taking the union bound we can obtain the conclusion.\qed
		
		(6). Let $r$ be an element of $[c]$, and  let $X_r$ be the random variable which counts the number of edges colored $r$ in $H$. It's easy to see $X_r\sim Bin(e(H),{1\over c})$ and $e(H)\sim Bin({n\choose k},p)$. Therefore by Theorem \ref{Theorem 2.1}  we can get w.h.p. $e(H)\leq2{n\choose k}p\leq {4n\log n\over k}\leq 2n\log n$, and
		\begin{align*}
		c\cdot\Pr[X_r\geq 10\log n]&\leq c{2n\log n\choose 10\log n}\left({1\over c}\right)^{10\log n}\\
		&\leq c\left({2en\log n\over 10c\log n}\right)^{10\log n}\\
		&\leq c\left({2e\over 10(1+\varepsilon)}\right)^{10\log n}=o(1).
		\end{align*}
		\qed
		
		(7). For any vertex set $X\subseteq [n]$ of size $t\leq{n\over(\log n)^{1\over k-1.5}}$,  the probability of $e_H(X)> 6|X|$ is  at most
		\begin{align*}
		{{t\choose k}\choose 6t}p^{6t}
		\leq\left({et^{k-1}p\over k!6}\right)^{6t}
		\leq\left({e\over k3}\log n\left({t\over n}\right)^{k-1}\right)^{6t}.
		\end{align*}
		Hence taking the union bound, the probability is at most
		\begin{align*}
		&\sum_{0<t\leq{n\over(\log n)^{1\over k-1.5}}}\left({en\over t}\right)^t\left({e\over k3}\log n\left({t\over n}\right)^{k-1}\right)^{6t}\\
		=&\sum_{0<t\leq{n\over(\log n)^{1\over k-1.5}}}\left({en\over t}\right)^t\left({e\over k3}\log n\left({t\over n}\right)^{k-1.5}\left({t\over n}\right)^{0.5}\right)^{6t}\\
		\leq&\sum_{0<t\leq{n\over(\log n)^{1\over k-1.5}}}\left({en\over t}\right)^t\left({e\over k3}\left({t\over n}\right)^{0.5}\right)^{6t}\\
		\leq&\sum_{0<t\leq{n\over(\log n)^{1\over k-1.5}}}\left[e\left({e\over k3}\right)^6\left({t\over n}\right)^2\right]^{t}
		=o(1).
		\end{align*}
		\qed
		
		(8). Fixed a vertex set $X\subseteq[n]$ of size $t\geq{n\over(\log n)^{1\over k-1.5}}$,  we have
		$${|X|^k\over k!}p\left( {n\over|X|}\right) ={|X|\over k}\log n\left( {|X|\over n}\right) ^{k-1}\left( {n\over|X|}\right) ={|X|\over k}\log n\left( {|X|\over n}\right) ^{k-2}\geq{|X|\over k}(\log n)^{1\over 2k-3}.$$
		Since  $e_H(X)\sim Bin({|X|\choose k}, p)$,  by Theorem \ref{Theorem 2.1}, there is
		$$\Pr\left[e_H(X)\geq {|X|^k\over k!}p\left( {n\over|X|}\right) \right]\leq\exp\left( -{3\over 16}\cdot{|X|^k\over k!}p\left( {n\over|X|}\right)\right)\leq \exp\left(-{3\over 16}\cdot{|X|\over k}(\log n)^{1\over 2k-3}\right) $$
		Hence taking the union bound, the probability is at most
		\begin{align*}
		&\sum_{{n\over (\log n)^{1\over k-1.5}}\leq t}{n\choose t}\exp\left( -{3\over 16}\cdot{t\over k}(\log n)^{1\over 2k-3}\right) \\
		&\leq\sum_{{n\over (\log n)^{1\over k-1.5}}\leq t}\left({en\over t}\right)^t\exp\left( -{3\over 16}\cdot{t\over k}(\log n)^{1\over 2k-3}\right) \\
		&\leq\sum_{{n\over (\log n)^{1\over k-1.5}}\leq t}\exp\left(t\log {en\over t}\right)\exp\left( -{3\over 16}\cdot{t\over k}(\log n)^{1\over 2k-3}\right) =o(1).
		\end{align*}
	\end{proof}
	The above Lemma's (2), (3) and (4)  shows that the order of vertex set $SMALL$ is 'small', and the  edge set $E\text{-}SMALL$ is rainbow, and the distance between any two vertices of $SMALL$ are not too close,  respectively. Next, we give that each $H\sim H_c^k(n,p)$ is w.h.p. has a 'lang' rainbow Berge path.

	\begin{definition}
		Let $g=g(n)$ be a positive integer. An edge-colored $k$-uniform
		hypergraph $H$ on $n$ vertices is called $g$-rainbow-pseudorandom if,
		for every pair of disjoint sets $A,B\subseteq V(H)$ with
		$|A|=|B|=g$,
		$c(E_H(A,B))\geq n$.
	\end{definition}

	The following Theorem is a generalization of  Lemma 4.4 \cite{BenEliezerKrivelevichSudakov2012}.
	
	\begin{theorem}\label{ Theorem 2.4} Let $g$ be a positive integer and let $H'$ be an edge-colored
		$k$-uniform hypergraph on $n$ vertices. If $H'$ is
		$g$-rainbow-pseudorandom, then $H'$ contains a rainbow Berge path
		of length at least $n-2g$.
	\end{theorem}

	\begin{proof}
		Algorithm. (Depth-First Search)
		
		(1) Arbitrarily choose an edge $e_1\in E(H')$ and two vertices $v_0, v_1\in e_1$. Initialize $P_1=v_0e_1v_1$, $A_1=V(H')\setminus \{v_0,v_1\}$, $B_1=\phi$,  $E_1=E_{H'}(v_1)$.
		
		(2) Choose a pair $(e_{i+1},v_{i+1})$ that unchosen before satifies the edge  $e_{i+1}\in E_{H'}(v_i)\setminus P_i$ and  vertex $v_{i+1}\in e_{i+1}\cap A_i$. If there are no remaining pairs, then proceed to step (6).
		
		(3) Test if $C(e_{i+1})\notin C(P_i)$.
		
		(4) If yes, then set  $ P_{i+1}=P_ie_{i+1}v_{i+1}, A_{i+1}=A_i\setminus\{V_{i+1}\}, B_{i+1}=B_i$.
		
		(5) If no, then  proceed to step (2).
		
		(6) Set $ P_{i+1}=P_i\setminus\{e_i,v_i\}, A_{i+1}=A_i, B_{i+1}=B_i\cup \{v_i\}$, proceed to step (2).

		According to the definitions of $A_i, B_i$, it is not difficult to obtain $c(E_H'(A_i, B_i) )<n$. Therefore,  at some steps $i$, we must have $|B_i|=g$ and  $|A_i|<g$ since   $|B_i|$ does not decrease and $|A_i|$  does not increase in this algorithm. Otherwise,  it contradicts the definition of $g$-rainbow-pseudorandom. Based on this, $l(P_i)+1=n-|A_i|-|B_i|\geq n-2g+1$.
	\end{proof}
	\begin{lemma}\label{ Lemma 2.5} Let integer $k\geq 3$, and let $0<\mu<\varepsilon<1$ be  constants. If $c=(1+\varepsilon)n$ and ${(k-1)!\log n\over n^{k-1}}\leq p\leq2{(k-1)!\log n\over n^{k-1}}$, then  edge- colored random $k$-uniform hypergraph $H\sim H_c^k(n,p)$ is w.h.p. satifies the following result. For every subset $C'\subseteq[c]$ of size $|C'|\geq(1+\mu)n$, the subhypergraph $H[[n];C']\subseteq H$ is ${n\over (\log n)^{0.49}}$-rainbow-pseudorandom.
	\end{lemma}
	\begin{proof} Following the definition of $g$-rainbow-pseudorandom. Suppose toward a contradiction that exists a color set $C'\subseteq[c]$ of size $|C'|\geq(1+\mu)n$, the hypergraph $H[[n];C']\subseteq H$ is not ${n\over (\log n)^{0.49}}$-rainbow-pseudorandom. That means there are two disjoint subsets $X$ and $Y$ of size ${n\over (\log n)^{0.49}}$, and a color set $C''\subseteq C'$ of size $rn(r\leq \varepsilon-\mu)$,  such that  none of the colors of $C''$ appearing on the edges between $X, Y$  in $H[[n];C']\subseteq H$. There is
		\begin{align*}
		&\Pr[none\ of\ the\ colors\ of\ C''\ appearing\ on\ the\ edges\ between\ X,\ Y\ in \ H[[n];C']\subseteq H]\\
		&\leq\left( 1-p+p\left( 1-{rn\over |C'|}\right) \right)^{e_{H}(X,Y)}\leq\left( 1-p+p\left( 1-{rn\over c}\right) \right) ^{|X||Y|{n-|X|-|Y|\choose k-2}}\\
		&\leq e^{-{prn\over c}|X||Y|(1+o(1)){n\choose k-2}}\leq e^{-{r\over c}(k-1)n^2(\log n)^{0.02}}\\
		&\leq e^{-{r\over (1+\varepsilon)}r(k-1)n(\log n)^{0.02}}.
		\end{align*}
		Since ${c\choose |C'|}{|C'|\choose |C''|}{n\choose |X|}{n\choose |Y|}\leq 2^{5n}$, taking the union bound we can get the result.
		
	\end{proof}
	Note that the order of $C'$ in this Lemma only needs to be slightly higher than $n$, without reaching  the value of $c$ in Theorem \ref{Theorem 1.1}. In fact, we use the Theorem \ref{ Theorem 2.4} for the subhypergraph instead of directly use it for $H\sim H_c^k(n,p)$ in the proof of Theorem \ref{Theorem 1.1}.
	
	\section{Expander $k$-uniform hypergraph}\label{Section 3}
	The Chapter 3  shows that the  edge set $E\text{-}SMALL$ is rainbow, and the  edge-colored random $k$-uniform hypergraph $H\sim H_c^k(n,p)$ is w.h.p. contains a 'lang' rainbow Berge path. Next,   we randomly select two  edges that incident to $v$ for each $v\in SMALL$. And let $E_S$ be the union of  the edges  selected for all vertices of $SMALL$. Obviously,  $E_S$ is rainbow. Let color subset $C_S:=C(E_S)$ and let vertex subset $V_S:=V(E_S)$.
	
	In this Chapter, we consider   a 'good' vertex subset disjoint with $V_S$ and two 'good' non-intersect  color subsets disjoint with  $C_S$  in  $H$, in which color sets and corresponding vertex sets form  subgraphs of $H$ is 'good' expander. In addition, this is related to the  'long' rainbow Berge path. Let's first introduce the definition of expander graph and some related properties.

	\begin{definition}\label{Definition 3.1}Let $r$ be a positive integer and let $m>0$.  A $k$-uniform hypergraph $H$  is called an $(r, m)$-$expander$ if for every  nonempty set $Y\subseteq V(H)$ with $|Y|\leq r$, there is
		$$|N_H(Y)\setminus Y|\geq m(k-1)|Y|.$$
	\end{definition}

	The following Lemmas  is a generalization of  Claim 2.8 \cite{BenShimonFerberHefetzKrivelevich2010} and Lemma 2.18 \cite{FerberKrivelevich2016}.
	
	\begin{lemma}\label{Lemma 3.2} Let $k\geq 3$, $m\geq 1$ and  $r\ge1$ be  integers. let $H$ be a $k$-uniform hypergraph. If there exists a vertex partition $V(H)=U\cup (V(H)\setminus U)$ such that  (1) $d_H(u)\geq m$ for every $u\in U$; (2)
		there is no Berge path of length at most $4$ with two distinct
		endpoints in $U$, and no Berge cycle of length at most $4$
		meeting $U$; (3)	 $H_2:=H[V(H)\setminus U]$ is an $(r,m+1)$-expander, then $H$ is an $(r,m)$-expander.
	\end{lemma}
	\begin{proof}Let $X\subseteq V(H)$ with $|X|\leq r$.  Let $X_1=X\cap U$, $X_2=X\setminus X_1$, observe that $|X|=|X_1|+|X_2|$. Since there is no path of length at most 4 in $H$ whose endpoints(might be the same) lie in $U$, hence $N_H(X_1)\cap X_1=\phi$, and $N_H(X_1)$  contains at most one vertex from each set $\{\{t\}\cup N_{H}(t)\}_{t\in X_2}$(So $|N_H(X_1)\cap(X_2\cup N_H(X_2))|\leq |X_2|$). Therefore,
		
		$$	N_H(X)\setminus X=(N_H(X_1)\setminus X)\cup (N_H(X_2)\setminus X)\supseteq (N_H(X_1)\setminus X_2)\cup (N_{H_2}(X_2)\setminus X_2),$$
		and
		\begin{align*}
		|N_H(X)\setminus X|&\geq |N_H(X_1)|+|N_{H_2}(X_2)\setminus X_2|-|N_H(X_1)\cap(X_2\cup N_{H_2}(X_2))|\\&\geq |N_H(X_1)|+|N_{H_2}(X_2)\setminus X_2|-|X_2|\\&\geq m(k-1)|X_1|+(m+1)(k-1)|X_2|-|X_2|\\&=m(k-1)|X|.
		\end{align*}
	\end{proof}
	\begin{definition}\label{Definition 3.3.} A $k$-$uniform\ bipatite\ hypergraph\ H[S,{W\choose k-1}]$ is a $k$-uniform hypergraph where vertex set is partitioned into two disjoint sets $S$ and $W$ such that each edge contains exactly one vertex from $S$ and $k-1$ vertices from $W$.
	\end{definition}
	
	\begin{definition}\label{Definition 3.4.}  Let $m$ be a positive constant, and let $H[S,{W\choose k-1}]$ be a $k$-uniform bipatite hypergraph.  A $m$-$matching$ $M$ from $S$ to $W$ of  $H[S,{W\choose k-1}]$ is a   bipartite subhypergraph   such that $d_M(s)=m$ for every $s\in S$ and $d_M(w)\leq1$ for every $w\in W$.
	\end{definition}
	\begin{lemma}\label{Lemma 3.5.} Let  $k\geq 3$, $m>3$ and  $r\ge 1$ be  integers, let $H$ be $k$-unifrom hypergraph. If there exists a vertex partition $V(H)=S\cup (V(H)\setminus S)$ such that    there is a $m$-matching from $S$ to $V(H)\setminus S$ and  $H[V(H)\setminus S]$ is an $(r, m)$-expander. Then, $H$ is an $(r, {m-1\over 2})$-expander.
	\end{lemma}
	\begin{proof}
		According to the definition of expander, we need to proof that for every $X\subseteq V(H)$ of size at most $r$, there is $|N_H(X)\setminus X|\geq {m-1\over 2}|X|$. Indeed, if $|X\cap S|\leq {|X|\over 2}$, since $H[V(H)\setminus S] $ is a $(r, m)$-expander, there is
		\begin{align*}
		|N_H(X)\setminus X|\geq &|N_H(X\setminus S)\setminus X| 
		\geq |N_H(X\setminus S)\setminus (X\setminus S)|-|X\cap S|\\
		\geq &m(k-1)|X\setminus S|-|X\cap S|
		\geq m(k-1){|X|\over 2}- {|X|\over 2}\\ \geq&{m-1\over2}(k-1)|X|.
		\end{align*}
		
		If $|X\cap S|>{|X|\over 2}$, since there exists a $m$-matching from $S$ to $V(H)\setminus S$, we have
		\begin{align*}
		|N_H(X)\setminus X|&\geq|N_H(X\cap S)\setminus X |\geq|N_H(X\cap S)\setminus(X\cap S)|-|X\setminus S|\\
		&\geq m(k-1)|X\cap S|-|X\setminus S|\geq m(k-1){|X|\over 2}- {|X|\over 2}
		\\&\geq{m-1\over2}(k-1)|X|.
		\end{align*}
	\end{proof}

	Next, we find the  vertex sets and color sets mentioned above. This mainly applies the following lemma.
	
	\begin{lemma}[{Lov\'{a}sz Local Lemma \cite{AlonSpencer2000}}]\label{Lemma 3.6.} Let $\varLambda_1,\varLambda_2,\cdots ,\varLambda_N$ be events in an arbitrary space. A directed graph $D=(V,E)$ on  vertex set $V=\{1,2,\cdots,N\}$ is called  dependency digraph for these events if for each $1\leq i\leq N$, the event $\varLambda_i$ is mutually independent of all the events $\{\varLambda_j:(i,j)\notin E\}$. Suppose that $D=(V,E)$ is a dependency digraph for the above events and suppose there $a_1,\ldots,a_N\in[0,1)$ such that 
		$$\Pr\left[\varLambda_i\right]\leq a_i\prod_{(i,j)\in E}(1-a_j)$$ for every $i\in[N]$.
		Then \[\Pr\left[\bigcap _{i=1}^N\overline{\varLambda_i}\right]\geq\prod_{i=1}^n(1-a_i)>0.\]
		In particular, with positive probability, no event $\varLambda_i$ holds. 	
	\end{lemma}

	\begin{lemma}\label{Lemma 3.7.} Let integer $k\geq 3$, and let $0<\eta,\mu<\varepsilon<1$ be constants. If $c=(1+\varepsilon)n$ and ${(k-1)!\log n\over n^{k-1}}\leq p\leq2{(k-1)!\log n\over n^{k-1}}$. Then  edge-colored random $k$-uniform $H\sim H^{k}_c(n,p)$ is w.h.p. satisfies the  following result. Suppose that $H'\subseteq H$ satisfies  $|V(H')|=(1-o(1))n$, $|C(H')|\geq(1+{\varepsilon\over 2})n$ and  the following properties$: $
		
		{\rm(i)} $\eta\log n\leq\delta(H')\leq \varDelta(H')\leq 10\log n$.
		
		{\rm(ii)} For each $v\in V(H')$,  $d_{H'}^c(v)\geq d_{H'}(v)-2$.
		
		{\rm(iii)} Each color appears on at most $10\log n$ times in $H'$.
		
		\noindent Then there exist a vertex set $W\subseteq V(H')$ and two disjoint color sets $C_1, C_2\subseteq C(H')$ such that the following holds$:$
		
		{\rm(1)} $|W|=(1+o(1)){n\over \log\log n}$.
		
		{\rm(2)} $|C_1|=(1+o(1))\mu n$, $|C_2|=(1+o(1))\mu n$.

		{\rm(3)} For every $v\in W$,  $d_{H'}^{C_1}(v,{W\choose k-1})\in ({\mu\eta\log n\over 2(\log \log n)^{k-1}}, {2\mu\log n\over (\log\log n)^{k-1}})$.
		
		{\rm(4)} For every $v\in V(H')$,  $d_{H'}^{C_2}(v,{W\choose k-1})\in ({\mu\eta\log n\over 2(\log \log n)^{k-1}}, {2\mu\log n\over (\log\log n)^{k-1}})$.
		
		{\rm(5)} For every $x\in C(H')$, $x$ appears on at most ${100\log n\over (\log\log n)^k}$ edges in $H[W;C(H')]\subseteq H$.
	\end{lemma}
	\begin{proof}	
		Let $W\subseteq V(H')$ be a random vertex subset, which is obtained by randomly and independently selecting each $v\in V(H')$ with probability ${1\over\log\log n}$. Let $C_1,C_2\subseteq C(H')$ be two color subsets obtained by the following way. First, each element in $C(H')$ is randomly and independently determined whether it belongs to set $C_1\cup C_2$ with probability $2\mu$, and then which set this element belongs to with probability $1/2$.
		
		Let's consider the opposite events of each event in $\{(1), (2), (3), (4), (5)\}$.
		\begin{itemize}
			\item  Let $\varLambda_W$ denote the event $|W|\neq(1+o(1)){n\over \log\log n}$.
			\item  For each $i\in \{1,2\}$, let $\varLambda_{C_i}$ denote the event $|C_i|\neq(1+o(1))\mu n$.
			\item  For each vertex $v\in V(H')$ and each $i\in \{1,2\}$, let $\varLambda_{v,C_i}$ denote the event $$d^{C_i}_{H'}\left(v,{W\choose k-1}\right) \notin \left( {\mu\eta\log n\over 2(\log \log n)^{k-1}}, {2\mu\log n\over (\log\log n)^{k-1}}\right) .$$
			\item Let $\varLambda_x$ denote the event $x$ appears on at least ${100\log n\over (\log\log n)^k}$ edges in $H[W;C(H')]$.
		\end{itemize}
		Defintion $\mathcal{V}:=\{\varLambda_W, \varLambda_{C_i}, \varLambda_{v,C_i}, \varLambda_x: i\in\{1,2\}, v\in V(H'), x\in C(H')\}$ is  the collection  of the  above opposite events. Thus, the  events (1)-(5) all holds of Lemma \ref{Lemma 3.7.} is equivalent to that
		\begin{align}
		\Pr\left[ \overline{\varLambda_W}\cap \overline{\varLambda_{C_1}}\cap \overline{\varLambda_{C_2}}\cap \left(  \bigcap_{i\in\{1,2\},v\in V(H')}\overline{\varLambda_{C_i}}\right) \cap \left( \bigcap_{x\in C(H')}\overline{\varLambda_x}\right) \right]>0.    \label{1}
		\end{align}
		
		We hope use Lemma \ref{Lemma 3.6.} to prove that inequality \eqref{1} holds. Let's first define a dependency graph $D$, where the vertex set is $\mathcal{V}$ and the edge set is as follow:
		\begin{itemize}
			\item all pairs $(\varLambda_W,\varLambda)$, where $\varLambda\in \mathcal{V}$, and
			\item all pairs $(\varLambda_{C_i},\varLambda)$, where $\varLambda\in \mathcal{V}$, and
			\item all pairs $(\varLambda_{v, C_i},\varLambda_{u, C_j})$\footnote{ the number of events $\varLambda_{u, C_j}$ which are neighbors of $\varLambda_{v, C_i}$ in $D$ is at most $k\varDelta(H')\cdot  k\varDelta(H')+\varDelta(H')\cdot 10\log n\cdot k\cdot =100(k^2+k)(\log n)^2$.}, where $v=u$ and $i\neq j$, or $v\neq u $ and $(\{v\}\cup N_{H'}(v) )\cap (\{u\}\cup N_{H'}(u) )\neq \phi $, or there exists a color $x\in C(H')$ appears on the edges that incident to both $u$ and $v$,  and
			\item all pairs $(\varLambda_{v, C_i},\varLambda_x)$\footnote{ the number of events  $\varLambda_x$ which are neighbors of $\varLambda_{v, C_i}$ in $D$ is at most $\varDelta(H')=10\log n$},$(\varLambda_x,\varLambda_{v, C_i})$\footnote{ the number of events $\varLambda_{v, C_i}$ which are neighbors of $\varLambda_x$ in $D$ is at most $10\log n\cdot k\varDelta(H')\cdot k=100k
				^2(\log n)^2$.}, where there exists an edge $e$ with color $x$ such that $e\cap (\{v\}\cup N_{H'}(v) )\neq \phi $, and
			\item all pairs $(\varLambda_x,\varLambda_y)$\footnote{the number of events $\varLambda_x$ which are neighbors of $\varLambda_y$ in $D$ is at most $10\log n\cdot k\varDelta(H')=100k
				(\log n)^2$.}, where there exists two edges $e$ with colors $x$ and $f$ with color $y$,   such that $e\cap f\neq \phi$.
		\end{itemize}
		Next, let's estimate the probability of each event in $\mathcal{V}$. Note that $\varLambda_W$, $\varLambda_{C_i}$ and $\varLambda_{v, C_i}$ obey binomial distribution, by using Theorem \ref{Theorem 2.1} we have
		
		(I) $\Pr[\varLambda_W]=\exp(-\Theta({n\over \log\log n}))$, define $a_W=\sqrt{\Pr[\varLambda_W]}=\exp(-\Theta({n\over 2\log\log n}))=o(1)$, and
		
		(II) $\Pr[\varLambda_{C_i}]=\exp(-\Theta(\mu n))$, define $a_{C_i}=\sqrt{\Pr[\varLambda_{C_i}]}=\exp(-\Theta({\mu\over 2} n))=o(1)$, and
		
		(III) $\Pr[\varLambda_{v,C_i}]=\exp(-\Theta({\mu d_{H'}(v)\over(\log\log n)^{k-1}}))\leq \exp(-\Theta({\mu\eta\log n\over(\log\log n)^{k-1}}))$, define $a_{v,C_i}=a_1\exp(-\Theta({\mu\eta\log n\over(\log\log n)^{k-1}}))$, in which $a_1$ be a constant.
		
		Finally, let's estimate $\Pr[\varLambda_x]$. Let $H'^x=H'[V(H'),x]$ be  the spanning subgraph  of $H'$ with color $x$, then $e(H'^x)\leq 10\log n$. Since  $d_{H'}^c(v)\geq d_{H'}(v)-2$ for every $v\in V(H
		')$, it follows that $d_{H'^x}(v)\leq 3$ for every $v\in V(H'^x)$. Shannon's theorem \cite{FioriniWilson1977} gives $\chi (H'^x)\leq {3\over2}\varDelta(H'^x)\leq {9\over 2}$. Thus  $H'^x$ has five matchings, and each of size at most $10\log n$, which ensures that any edge selected in the same matching is independent. If $\varLambda_x$ holds, then at least one of the five matchings has at least ${20\log n\over (\log\log n)^k}$ edges. For a fixed edge $e$ with color $x$, the probability that $e\subseteq {W\choose k}$ is ${1\over (\log\log n)^k}$. Hence by using Theorem \ref{Theorem 2.1} we have
		
		(IV) $\Pr[\varLambda_x]=\exp(-{3\over 16}\cdot{20\log n\over (\log\log n)^k})$, define $ a_x=a_2\Pr[\varLambda_x]\leq a_2\exp(-{\log n\over 3(\log\log n)^k})$, in which $a_2$ be a constant.

		Therefore, according to Lemma \ref{Lemma 3.6.}, the sufficient condition for the inequality (1) to hold is that the following inequality holds.
		\begin{align*}
		\Pr[\varLambda_W]
		\leq a_W(1-a_{C_1})(1-a_{C_2})(1-a_{v,C_i})^{2n}(1-a_x)^{10\log n},
		\end{align*}
		\begin{align*}
		\Pr[\varLambda_{C_i}]\leq a_{C_i}(1-a_W)(1-a_{C_j})(1-a_{v,C_i})^{2n}(1-a_x)^{10\log n},
		\end{align*}
		\begin{align*}
		\Pr[\varLambda_{v,C_i}]
		\leq a_{v,C_i}(1-a_W)(1-a_{C_1})(1-a_{C_2})(1-a_{u,C_j})^{100(k^2+k)(\log n)^2}(1-a_x)^{10\log n},
		\end{align*}
		\begin{align*}
		\Pr[\varLambda_x]
		\leq a_x(1-a_W)(1-a_{C_1})(1-a_{C_2})(1-a_{v,C_i})^{100k
			^2(\log n)^2}(1-a_x)^{100k
			(\log n)^2}.
		\end{align*}
		Recall the definitions of  $a_W$, $a_{C_i}$,
		$a_{v,C_i}$ and $a_x$. The inequalities above becomes
		\begin{align*}
		\sqrt{\Pr[\varLambda_W]}
		\leq (1-o(1))(1-a_{v,C_i})^{2n}(1-a_x)^{10\log n},
		\end{align*}
		\begin{align*}
		\sqrt{\Pr[\varLambda_{C_i}]}\leq (1-o(1))(1-a_{v,C_i})^{2n}(1-a_x)^{10\log n},
		\end{align*}
		\begin{align*}
		1
		\leq a_1(1-o(1))(1-a_{u,C_j})^{100(k^2+k)(\log n)^2}(1-a_x)^{10\log n},
		\end{align*}
		\begin{align*}
		1
		\leq a_2(1-o(1))(1-a_{v,C_i})^{100k
			^2(\log n)^2}(1-a_x)^{100k
			(\log n)^2}.
		\end{align*}
		According to the fact  that $ e^{-2x}\leq 1-x$, we only need to   prove that
		\begin{align*}
		-\Theta\left( {n\over 2\log\log n}\right) \leq\left(-2\cdot 2n\cdot a_1\exp\left( -\Theta\left( {\mu\eta\log n\over(\log\log n)^{k-1}}\right) \right) \right)+\left(-2\cdot 10\log n\cdot a_2\exp\left( -{\log n\over 3(\log\log n)^k}\right) \right),
		\end{align*}
		\begin{align*}-\Theta\left( {\mu\over 2} n\right) \leq\left(-2\cdot 2n\cdot a_1\exp\left( -\Theta\left( {\mu\eta\log n\over(\log\log n)^{k-1}}\right) \right) \right)+\left(-2\cdot 10\log n\cdot a_2\exp\left( -{\log n\over 3(\log\log n)^k}\right) \right),
		\end{align*}
		\begin{align*}0\leq \log a_1+\left(-2\cdot100(k^2+k)(\log n)^2\cdot a_1\exp\left( -\Theta\left( {\mu\eta\log n\over(\log\log n)^{k-1}}\right) \right) \right)+\left(-2\cdot10\log n \cdot a_2\exp\left( -{\log n\over 3(\log\log n)^k}\right) \right),
		\end{align*}
		\begin{align*}0\leq \log a_2+\left(-2\cdot100k
		^2(\log n)^2\cdot a_1\exp\left( -\Theta\left( {\mu\eta\log n\over(\log\log n)^{k-1}}\right) \right) \right)+\left(-2\cdot100k
		(\log n)^2\cdot a_2\exp\left( -{\log n\over 3(\log\log n)^k}\right) \right).
		\end{align*}
		In fact, the  first inequality  holds since $\log\log n\ll{\log n\over (\log\log n)^{k-1}}\leq\exp\left( {\log n\over (\log\log n)^{k-1}}\right)$ so that
		\[-\left( {1\over \log\log n}\right)\ll -\exp\left( -\left( {\log n\over(\log\log n)^{k-1}}\right) \right) ,\]
		and the second inequality holds when $n$ is large enough and $\mu$ are chosen so that
		\[-\mu\ll-\exp\left( -\left( {\mu\eta\log n\over(\log\log n)^{k-1}}\right) \right) .\]
		Since $(\log n)^2\cdot \exp({-\log n\over(\log\log n)^{k-1}})=\exp({2\log\log n})\cdot \exp({-\log n\over(\log\log n)^{k-1}})=o(1)$, the third and fourth inequalities  holds for some constants $a_1$ and $a_2$.
	\end{proof}
	
	We next  consider the roles of vertex set $W$ and colors set $C_1$ and $C_2$ found.  We hope to find a 'good' expander subgraph which the coloring of edge from $C_1$ and $C_2$.
	\subsection{Edge colored expander subhypergraph with color set $C_1$  }\label{Section 3.}
	In this subsection, we prove the existence of an expander subhypergraph where the vertex set from $W$ and the coloring  from $C_1$.
	
	\begin{lemma}\label{Lemma 3.8.} Let integer $k\geq 3$, and let $0< \eta,\mu< \varepsilon<1$. If ${(k-1)!\log n\over n^{k-1}}\leq p\leq2{(k-1)!\log n\over n^{k-1}}$ and $c=(1+\varepsilon)n$, then  edge colored random $k$-uniform hypergraph $H\sim H_c^{k}(n,p)$ is w.h.p. such that the following properties hold. Suppose that $W\subseteq[n]$  is of size $(1+o(1)){n\over\log\log n}$ and $C_1\subseteq [c]$ is of size $(1+o(1))\mu n$, let $H_1:=H[W;C_1]\subseteq H$.
		Then for every subset $X\subseteq W$,
		
		{\rm(1)}  if $|X|\leq {n\over(\log n)^{1\over k-1.5 }}$, then $e_{H_1}(X)\leq 6|X|$, and
		
		{\rm(2)} if $ {n\over(\log n)^{1\over k-1.5}}\leq|X|\leq|W|$, then $e_{H_1}(X)\leq {|X|^k\over k!}p({n\over|X|})$, and
		
		{\rm(3)} if $|X|\geq{n\over (\log n)^{1\over k+1}}$, then $e_{H_1}(X)\leq \mu{|X|^k\over k!}p$.
	\end{lemma}
	\begin{proof}	
		(1) and (2) are hold by applying $e_{H_1}(X)\leq e_{H}(X)$ to (7) and (8) of Lemma \ref{Lemma 2.2}  Fixed a subset $X\subseteq W$ of size $|X|\geq{n\over \log^{1\over k+1}n}$, since $e_{H_1}(X)\sim Bin({|X|\choose k}, \mu p)$, we have
		$${\rm E}(e_{H_1}(X))=\mu{|X|\choose k}p=(1-o(1))\mu\left({|X|^k\over k!}\right)p\geq(1-o(1)){1\over k}\mu n(\log n)^{1\over k+1}.$$
		Then by Theorem \ref{Theorem 2.1}, we have
		\[{n\choose |X|}\Pr\left[e_{H_1}(X)> \mu{|X|^k\over k!}p\right]\leq 2^ne^{-o(1){\rm E}(e_{H_1}(X))}=o(1).\]
	\end{proof}
	
	\begin{theorem}\label{Theorem 3.9.} Let integer $k\geq 3$, and let $0< \eta,\mu< \varepsilon<1$. If $c=(1+\varepsilon)n$ and ${(k-1)!\log n\over n^{k-1}}\leq p\leq2{(k-1)!\log n\over n^{k-1}}$, then   $H\sim H_c^k(n,p)$ is w.h.p. satisfying the following properties. Suppose that $W\subseteq V(H)$ is of size $(1+o(1)){n\over \log\log n}$, $C_1\subseteq [c]$ is of size $(1+o(1))\mu n$, and $H_1:=H[W;C_1]$ is a subhypergraph of $H$ satisfying properties (3) and (5) of Lemma \ref{Lemma 3.7.} and properties (1)-(3) of Lemma \ref{Lemma 3.8.}. Then there exists an integer $m_0=m_0(k,\eta,\mu)$ such that,
		for every integer $m\ge m_0$, the hypergraph $H_1$ contains a
		spanning subhypergraph $H_1'$ satisfying the following results $:$
		
		{\rm(1)} $H_1'$ is rainbow.
		
		{\rm(2)}	$e(H_1')\leq m|W|$.
		
		{\rm(3)}	$H_1'$ is an $({\mu\eta|W|\over 100},100)$-expander.
	\end{theorem}
	\begin{proof}
		Let $m$ be an enough large  interger. Suppose that $H_1$ satisfies properties (3) and (5) of Lemma \ref{Lemma 3.7.}  and properties (1)-(3) of Lemma \ref{Lemma 3.8.} For every $w\in W$, independently,  randomly  and repeatable choose $m$ edges which incident to $w$ in $H_1$, and  let $\mathcal {E}(w)$ be the set of the edges choose by $w$. Let $H_1'$ be the subgraph with edge set $\cup_{w\in W}\mathcal{E}(w)$, then $e(H_1')\leq m|W|$. We next claim that $H_1'$ is w.h.p satisfies (1) and (3). Similar to  Lemma \ref{Lemma 3.7.}, we will  prove it by  Lemma \ref{Lemma 3.6.}
		
		Let's consider the following two events. First, for any two edges $e_1$ and $e_2$ of same color, let $\varLambda (e_1,e_2)$ denote the event that edges $e_1 \neq  e_2$  are all chosen  or edges $e_1=e_2$ is chosen twice. Define
		\[\varLambda:=\{\varLambda(e_1,e_2):e_1\ and\ e_2\ have\ the\ same\ color\}.\]
		Obviously, if none of the events in $\varLambda$ happens, then properties (1) hold.
		
		Second,  let $T=[{n\over 101(k-1)(\log n)^{1\over k-1.5}},{\mu\eta|W|\over 100}]$. For every $t\in T $,  let $X\subseteq W$ be a set of size $101(k-1)t$, denote by $\varGamma_t(X)$ be the event that $$e_{H_1'}(X)\geq {m\over 101(k-1)}|X|\geq mt.$$
		Define $\varGamma_t=\{\varGamma_t(X): X\subseteq W, |X|=101(k-1)t\}$.
		
		\noindent\textbf {Claim 1.} \label{Claim 1}
		If none of the events in $\varLambda$ happens and none of the events $\varGamma_t$ happens for every $t\in  T$, then property (3) holds.
		\begin{proof}
			Assume otherwise, there exists a subset $X'\subseteq W$ of size at most ${\mu\eta|W|\over 100}$ satisfies $|N_{H_1'}(X')\cup X'|< 101(k-1)|X'|$. Since none of the events in $\varLambda$ happens, and according to the selection method of $\mathcal {E}(v)$ for every $v\in  X'$, it follows that $$e_{H_1}(N_{H_1'}(X')\cup X')\geq e_{H_1'}(N_{H_1'}(X')\cup X')\geq m|X'|>{m\over 101(k-1)}|N_{H_1'}(X')\cup X'|.$$ If $|N_{H_1'}(X')\cup X'|\leq {n\over(\log n)^{1\over k-1.5}}$, then it contradicts to (1) of Lemma 3.8 since $m$ be an enough   large  interge. If $ {n\over(\log n)^{1\over k-1.5}}\leq|N_{H_1'}(X')\cup X'|\leq 101(k-1){\mu\eta|W|\over 100}$, then it contradicts to the event $\varGamma_t(N_{H_1'}(X')\cup X')$, in which $t={|N_{H_1'}(X')\cup X'|\over 101(k-1)}$.
		\end{proof}
		
		Based on the events described above, for prove the results (1)-(3) of Theorem \ref{Theorem 3.9.}, we only need to prove  that
		\begin{align}
		\Pr\left[ \left( \bigcap_{C(e_1)=C(e_2)}\overline{\varLambda(e_1,e_2)}\right)\cap  \left( \bigcap_{t\in T,X\subseteq W,|X|=101(k-1)t}\overline{\varGamma_t(X)}\right)\right]>0. \label{(2)}
		\end{align}
		Define a dependence graph $D$ with vertex set $\mathcal{V}:=\varLambda\cup \left( \bigcup_{t\in T}\varGamma_t\right)$, and edge set is as follow:
		\begin{itemize}
			\item all pairs $(\varLambda(e_1,e_2), \varLambda(e_3,e_4))$\footnote{Since  each color class contains at most ${100\log n\over (\log\log n)^k}$ edges,  the number of events  $\varLambda(e_3,e_4)$ which are neighbors of $\varLambda(e_1, e_2)$ in $D$ is at most $2k\varDelta(H_1){100\log n\over (\log\log n)^k}\leq{2000k(\log n)^2 \over (\log\log n)^{2k-1}}$.}, where $\varLambda(e_1,e_2),\varLambda(e_3,e_4)\in \mathcal{V}$ and $(e_1\cup e_2)\cap (e_3\cup e_4)\neq \phi$, and
			\item all pairs $(\varLambda(e_1,e_2), \varGamma_t(X))$, where $\varLambda(e_1,e_2),\varGamma_t(X)\in \mathcal{V}$, and
			\item all pairs $( \varGamma_t(X),\varLambda(e_1,e_2))$\footnote{Since each color class contains at most ${100\log n\over (\log\log n)^k}$ edges, the number of events  $\varLambda(e_1,e_2)$ which are neighbors of $\varGamma_t(X)$ in $D$ is at most $|X|\varDelta(H_1){100\log n\over (\log\log n)^k}\leq|X|{2000(\log n)^2 \over (\log\log n)^{2k-1}}$.}, where $\varLambda(e_1,e_2),\varGamma_t(X)\in \mathcal{V}$ and $(e_1\cup e_2)\cap X\neq \phi$, and
			
			\item all pairs $(\varGamma_t(X),\varGamma_t(Y))$, where $\varGamma_t(X),\varGamma_t(Y)\in \mathcal{V}$.
		\end{itemize}
		\noindent\textbf {Claim 2.} For each $\varLambda(e_1,e_2)\in \varLambda $, $\Pr[\varLambda(e_1,e_2)]\leq\left(2km{(\log\log n)^{k-1}\over \mu\eta\log n}\right)^2$.	
		\begin{proof}
			It has  the following three cases:
			
			\noindent\textbf{Case 1.} If $e_1= e_2$, 
			then
			\[\Pr[\varLambda(e_1,e_2)]\leq {k\choose 2}\cdot\left({m\over \delta(H_1)}\right)^2\leq\left(2km{(\log\log n)^{k-1}\over \mu\eta\log n}\right)^2,\]
			
			\noindent\textbf{Case 2.} If $e_1\neq e_2, e_1\cap e_2=\phi$, then
			\[\Pr[\varLambda(e_1,e_2)]\leq\left(k\cdot{m\over \delta(H_1)}\right)^2\leq\left(2km{(\log\log n)^{k-1}\over \mu\eta\log n}\right)^2,\]
			
			\noindent\textbf{Case 3.}If $e_1\neq e_2, |e_1\cap e_2|=i\geq 1$, then
			\[\Pr[\varLambda(e_1,e_2)]\leq \left((k-i)\cdot {m\over \delta(H_1)}+i\cdot \left({m-1\over \delta(H_1)-1}\right)\right)^2\leq\left(2km{(\log\log n)^{k-1}\over \mu\eta\log n}\right)^2.\]
			Therefore,  $\Pr[\varLambda(e_1,e_2)]\leq\left(2km{(\log\log n)^{k-1}\over \mu\eta\log n}\right)^2$, and let $a=a_0\left(2km{(\log\log n)^{k-1}\over \mu\eta\log n}\right)^2$, in which $a_0$ be a constant.
		\end{proof}
		\noindent\textbf {Claim 3.} Let $T_1=[{n\over101(k-1)(\log n)^{1\over k-1.5}},{n\over 101(k-1)(\log n)^{1\over k+1}}], T_2=[{n\over 101(k-1)(\log n)^{1\over k+1}},{\mu\eta|W|\over 100}]$. If $t\in T_1$, then \[\Pr[\varGamma_t(X)]\leq\left((10k)^{3k}\left( {({t\over n})^{k-2}(\log\log n)^{k-1}\over \mu\eta}\right) \right)^{mt}.\]
		If $t\in T_2$, then
		\[\Pr[\varGamma_t(X)]\leq\left(10^{k}\left( {({t\over n})^{k-1}(\log\log n)^{k-1}\over \eta}\right) \right)^{mt}.\]
		\begin{proof}	
			For any two edges $e$ and $f$ in $E_{H'_1}(X)$, let $|e\cap f|=i_{ef}$, and define $i=\max_{e,f\in E_{H'_1}(X)}\{i_{ef}\}$. For every vertex $v\in W$, let $j_v$ denotes the number of edges containing vertex $v$ in $H'_1$, and define $j=\min_{v\in W}\{j_v\}$. There is
			\begin{align*}
			\Pr[\varGamma_t(X)]&\leq{e_{H_1}(X)\choose mt}\cdot \left((k-i)\cdot {m\over \delta(H_1)}+i\cdot {m-j\over \delta(H_1)-j}\right)^{mt}\\
			&={e_{H_1}(X)\choose mt}\cdot \left((k-i)\cdot m\left({\mu\eta\log n\over 2(\log\log n)^{k-1}}\right)^{-1}+i\cdot (m-j)\left({\mu\eta\log n\over 2(\log\log n)^{k-1}}-j\right)^{-1}\right)^{mt}\\
			&\leq {e_{H_1}(X)\choose mt}\cdot \left(2k m{(\log\log n)^{k-1}\over \mu\eta\log n}\right)^{mt}\leq \left({ee_{H_1}(X)\over mt} \right)^{mt} \cdot\left(2km{(\log\log n)^{k-1}\over \mu\eta\log n}\right)^{mt}\\
			&= \left({ee_{H_1}(X)\cdot 2k(\log\log n)^{k-1}\over t\cdot\mu\eta\log n}\right)^{mt} .
			\end{align*}
			
			If $t\in T_1$, then ${n\over(\log n)^{1\over k-1.5}}\leq|X|\leq{n\over (\log n)^{1\over k+1}}$, by \eqref{(2)} of Lemma \ref{Lemma 3.8.} we have $$e_{H_1}(X)\leq{|X|^k\over k!}p\left( {n\over|X|}\right)\leq2\left({101(k-1)t\over n }\right)^{k-1}{ n\log n\over k}.$$ From this, we get
			\[\Pr[\varGamma_t(X)]\leq\left((10k)^{3k}\left( {({t\over n})^{k-2}(\log\log n)^{k-1}\over \mu\eta}\right) \right)^{mt}.\]
			Let $y_t=e^{a_1t}\left((10k)^{3k}\left( {({t\over n})^{k-2}(\log\log n)^{k-1}\over \mu\eta}\right) \right)^{mt}$, in which $a_1$ be a constant. Note that
			\begin{align*}
			\sum_{t\in T_1}|\varGamma_t|\cdot y_t&=\sum_{t\in T_1} {|W|\choose 101(k-1)t}e^{a_1t}\left((10k)^{3k}\left( {({t\over n})^{k-2}(\log\log n)^{k-1}\over \mu\eta}\right) \right)^{mt}\\
			&\leq\sum_{t\in T_1} \left( {en\over 101(k-1)t\log\log n}\right) ^{101(k-1)t}e^{a_1t}\left((10k)^{3k}\left( {({t\over n})^{k-2}(\log\log n)^{k-1}\over \mu\eta}\right) \right)^{mt}\\
			&=\sum_{t\in T_1} \left[\left( {en\over 101(k-1)t\log\log n}\right) ^{101(k-1)}e^{a_1}\left((10k)^{3k}\left( {({t\over n})^{k-2}(\log\log n)^{k-1}\over \mu\eta}\right)\right) ^{m}\right]^t\\
			&=\sum_{t\in T_1} O_{k,\mu,\eta}\left[\left( {t\over n}\right)^{(k-2)m-101(k-1)}(\log\log n)^{(k-1)m-101(k-1)}\right]^t\\
			&\leq\sum_{t\in T_1} O_{k,\mu,\eta}\left[\left( {1\over (\log n)^{1\over k+1}}\right) ^{(k-2)m-101(k-1)}(\log\log n)^{(k-1)m-101(k-1)}\right]^t=o(1).
			\end{align*}
			If $t\in T_2$, then ${n\over (\log n)^{1\over k+1}}\leq |X|\leq 101(k-1){\mu\eta|W|\over 100}$, by (3) of Lemma \ref{Lemma 3.8.} we have
			$$e_{H_1}(X)\leq \mu{|X|^k\over k!}p\leq 2\mu(101(k-1))^k\left({t\over n}\right)^{k-1}{t\log n\over k}.$$ From this, we get
			\[\Pr[\varGamma_t(X)]\leq\left(10^{3k}\left( {({t\over n})^{k-1}(\log\log n)^{k-1}\over \eta}\right) \right)^{mt}.\]
			Let $y_t=e^{a_2t}\left(10^{3k}\left( {({t\over n})^{k-1}(\log\log n)^{k-1}\over \eta}\right) \right)^{mt}$, in which $a_2$ be a constant. Note that
			\begin{align*}
			\sum_{t\in T_2}|\varGamma_t|\cdot y_t&=\sum_{t\in T_2}{|W|\choose 101(k-1)t}e^{a_2t}\left(10^{3k}\left( {({t\over n})^{k-1}(\log\log n)^{k-1}\over \eta}\right) \right)^{mt}\\
			&\leq\sum_{t\in T_2} \left( {en\over 101(k-1)t\log\log n}\right) ^{101(k-1)t}e^{a_2t}\left(10^{3k}\left( {({t\over n})^{k-1}(\log\log n)^{k-1}\over \eta}\right) \right)^{mt}\\
			&\leq\sum_{t\in T_2} O_k\left( e^{a_2t}\left( {{t\over n}\log\log n\over \eta}\right) ^{(m-101)t(k-1)}\right) \leq\sum_{t\in T_2} O_k\left( e^{a_2t}\mu^{(m-101)t(k-1)}\right) =o(1),
			\end{align*}
			where the last inequality holds for some appropriate $\mu, a_2, m$.
		\end{proof}
		
		Therefore, according to Lemma \ref{Lemma 3.6.}, the sufficient condition for the inequality \eqref{(2)} to hold is that the following inequalities holds for every $\varLambda(e_1,e_2),\varGamma_t(X)\in \mathcal{V}$,
		
		\begin{align}
		\Pr[\varLambda(e_1,e_2)]\leq a\cdot(1-a)^{{2000k(\log n)^2 \over (\log\log n)^{2k-1}}}\cdot\prod_{t\in T}(1-y_t)^{|\varGamma_t|}, \label{3}
		\end{align}
		\begin{align}
		\Pr[\varGamma_t(X)]\leq y_t\cdot(1-a)^{|X|{2000(\log n)^2\over (\log\log n)^{2k-1}}}\cdot\prod_{t\in T}(1-y_t)^{|\varGamma_t|}.\label{4}
		\end{align}
		Recall the values of $x$  and $y_t$, and simplify with inequatlity $ e^{-2x}\leq1-x$ to \eqref{3} and \eqref{4}, this give
		\begin{align*}
		1\leq a_0 \cdot\exp\left({-2a_0\left( 2km{(\log\log n)^{k-1}\over \mu\eta\log n}\right) ^2\cdot{\left( {2000k(\log n)^2\over (\log\log n)^{2k-1}}\right) }}\right) ,\\
		1\leq e^{a_1t}\cdot\exp\left({-2a_0\left( 2km{(\log\log n)^{k-1}\over \mu\eta\log n}\right) ^2\cdot{101(k-1)t\left( {2000(\log n)^2\over (\log\log n)^{2k-1}}\right) }}\right),\\
		1\leq e^{a_2t}\cdot\exp\left({-2a_0\left( 2km{(\log\log n)^{k-1}\over \mu\eta\log n}\right) ^2\cdot{101(k-1)t\left( {2000(\log n)^2\over (\log\log n)^{2k-1}}\right) }}\right).
		\end{align*}
		The above inequalities holds for some  appropriate constants $a_0, a_1$ and $a_2$ since ${\log n\over \log\log n}$ is sufficiently large if $n$ is sufficiently large.
	\end{proof}
	
	\subsection{Edge colored expander subhypergraph with color set $C_2$  }\label{Section 3.2}
	In this subsection, we prove the existence of an expander subgraph where the coloring of edges  from $C_2$, particularly where the vertex set of the subhypergraph is related to  $W$ and  the 'long' rainbow Berge path found before.

	\begin{theorem} [{Aharoni and Haxell, \cite{AharoniHaxell2000}}]\label{Theorem 3.10.} Let $k$ and $m$ be positive integers and let $ H=\{ H_1,\cdots,H_t \}$ be a family $k$-uniform hypergraphs on the same vertex set. If for every $I\subseteq [t]$, the $k$-uniform hypergraph $\cup_{i\in I} H_i $ contains a matching of size greater than $mk(|I|-1)$, then there exists a function $f:[t]\times[m]\rightarrow \cup ^t_{i=1}E( H_i)$ such that $f(i,j)\in E( H_i)$ for every $(i,j)\in [t]\times[m]$, and $f(i,j)\cap f(i',j')=\phi$ if and only if  $(i,j)\neq(i',j').$
	\end{theorem}
	\begin{theorem} [{Haxell, \cite{Haxell1995}}]\label{Theorem 3.11.}A	$k$-uniform bipatite hypergraph $H:=H[I, {W\choose k-1}]$ contains a $d$-matching from $I$ to $W$ if and only if for any $X\subseteq I$, there is $$|N_H(X)|> d(2k-3)(|X|-1).$$	
	\end{theorem}
	
	\begin{theorem}\label{Theorem 3.12.} Let integers $k\geq 3, m>0$, and let $0<\mu<\varepsilon<1$ be contants. If $c=(1+\varepsilon)n$ and  ${(k-1)!\log n\over n^{k-1}}\leq p\leq2{(k-1)!\log n\over n^{k-1}}$, then edge-colored random $k$-uniform hypergraph $H\sim\mathcal{H}^k_c(n,p)$ is w.h.p. such that the following result holds. Suppose that
		
		{\rm(1)} $W\subseteq[n]$ with $|W|=(1+o(1)){n\over\log\log n}$, and
		
		{\rm(2)} $S\subseteq[n]$ with ${n\over\log^{0.4}n}\leq|S|\leq{2n\over\log^{0.4}n}$, and
		
		{\rm(3)} $C_2\subseteq [c]$ with $|C_2|=\mu n$, and
		
		{\rm(4)} for every $s\in S$, 	there are at least ${\log n\over(\log\log n)^k}$ edges $e$ incident to $s$ such that $|e\cap W|=k-1$ and $C(e)\in C_2$(so $e(H[s,{W\choose k-1};C_2])\geq {\log n\over(\log\log n)^k}$).
		
		\noindent Then, there exists a rainbow $m$-matching from $S$ to $W$ in $H[S,{W\choose k-1};C_2]\subseteq H$.
	\end{theorem}
	\begin{proof}
		For every $s\in S$, we construct a $k$-uniform bipartite hypergraph $\mathcal{H}_s$ with vertex set $C_2$ and $W$, edge set as follow
		$$E(\mathcal{H}_s):=\{e'=(e\cup \{C(e)\})\setminus \{s\}|  e\in E(H),\ \ s\in\ e\ and\ |e\cap M|=k-1,\ C(e)\in C_2\}.$$
		So there exists  a rainbow $m$-matching from $S$ to $W$, with all colors from $C_2$ in $H$, which means there exists a matching set  $\{M_s \}_{s\in S}$ in $\cup_{s\in S} E(\mathcal{H}_s)$, such that for every $s\in S$, $|M_s|=m$ and  each edge of $M_s$ is consist of one vertex of $C_2$ and $k-1$ vertices of $W$. That is, we need to find a function $f:S\times[m]\rightarrow \cup _{s\in S}E(\mathcal H_s)$ such that $f(i,j)\in E(\mathcal H_i)$ for every $(i,j)\in S\times[m]$, and $f(i,j)\cap f(i',j')=\phi$ for $(i,j)\neq(i',j')$.  Theorem 3.10 shows that its sufficient condition  is that for every $I\subseteq S$, the hypergraph $\cup_{i\in I}\mathcal H_i$ contains a matching of size greater than $mk(|I|-1)$.
		
		\noindent\textbf{Case 1:} If $|I|\leq{n\over(\log n)^{0.8}}$.
		
		\noindent \textbf{Claim 4}. There exists a $(\log n)^{0.3}$-matching from $I$ to $W$ in $H$.
		
		\begin{proof} Consider the $k$-uniform bipartite subhypergraph $B:=H[I,{W\choose k-1};C_2]$ of $H$. Following Theorem \ref{Theorem 3.11.}, we have  $B$ contains a $d$-matching from $I$ to $W$ if and only if for every $X\subseteq I$, there is  $$|N_B(X)|> d(2k-3)(|X|-1).$$
			
			Suppose that fails for $B$ with $d=(\log n)^{0.3}$, then there exists a subset $X\subseteq I$ with $|N_B(X)\cup X|\leq d(2k-3)|X|$. Let $Y=N_B(X)\cup X$, then $|Y|\leq d(2k-3)|X|\leq{(2k-3)n\over (\log n)^{0.5}}$, by (4) of Theorem \ref{Theorem 3.12.} we have  $$e_B(Y)\geq |X|{\log n\over (\log\log n)^k}\geq {|Y|\over d(2k-3)}{\log n\over (\log\log n)^k}\geq {|Y|(\log n)^{0.7}\over (2k-3)(\log\log n)^k}.$$
			It's contradiction to (7) and (8) of Lemma \ref{Lemma 2.2} since
			\begin{align*}
			e_B(Y)&\leq e_H(Y)\leq{|Y|^k\over k!}p\left( {n\over|Y|}\right) \leq{2\log n\over k}|Y|\left( {|Y|\over n}\right)^{k-2}\\
			&\leq{2\log n\over k}|Y|\left( {(2k-3)\over (\log n)^{0.5}}\right)^{k-2} \leq|Y|{2(2k-3)^{k-2}\over k}(\log n)^{0.5} .
			\end{align*}
		\end{proof}
		
		Next, we claim that for any $|I|\leq{n\over(\log n)^{0.8}}$, the number of distinct colors of $C_2$ for the above matching in Claim 1 is more than $mk|I|$( So  $k$-uniform hypergraph $\cup_{i\in I}\mathcal H_i$ contains a matching of size greater than $mk|I|$). Indeed,  if not,  then there is a collection of $|I|\leq{n\over (\log n)^{0.8}}$ stars, each of size $d=(\log n)^{0.3}$, and the number of  distinct colors of $C_2$ appearing  on its edges is at most $mk|I|$. However,
		\begin{align*}
		\Pr(exists\ a\ &collection\ of\ |I|\leq{n\over (\log n)^{0.8}}\ star\,\ each\ of\ size\ d,\ and\ the\ number\ of\ distinct\ colors\\ &\ of\ C_2\ appearing\ on\ its\ edges\ is\ at\ most\ km|I|)\\
		&\leq\sum_{|I|\leq{n\over (\log n)^{0.8}}}{n\choose |I|}{{n\choose k-1}\choose d}^{|I|}p^{|I|d}{|C_2|\choose km|I|}\left({km|I|\over |C_2|}\right)^{|I|d}\\
		&\leq\sum_{|I|\leq{n\over (\log n)^{0.8}}}\left({en\over |I|}\right)^{|I|}\left({en^{k-1}p\over d}\right)^{|I|d}\left({e\mu n\over km|I|}\right)^{km|I|}\left({km|I|\over \mu n}\right)^{|I|d}\\
		&=\sum_{|I|\leq{n\over (\log n)^{0.8}}}\left[\left({en\over |I|}\right)\left({e\mu n\over km|I|}\right)^{km}\left({en^{k-1}pkm|I|\over d\mu n}\right)^{d}\right]^{|I|}\\
		&\leq\sum_{|I|\leq{n\over (\log n)^{0.8}}}\left[\left({en\over |I|}\right)\left({e\mu n\over km|I|}\right)^{km}\left({2e k!m|I|\log n\over d \mu n}\right)^{d}\right]^{|I|}\\
		&=\sum_{|I|\leq{n\over (\log n)^{0.8}}}\left[O_{k,m,\mu}\left({|I|\over n}\right)^{d-1-km}\left({\log n\over d}\right)^{d}\right]^{|I|}\\
		&\leq\sum_{|I|\leq{n\over (\log n)^{0.8}}}\left[O_{k,m,\mu}\left({1\over (\log n)^{0.8}}\right)^{d-1-km}(\log n)^{0.7d}\right]^{|I|}
		=o(1),
		\end{align*}
		where the last inequality holds for  $d\gg km$.
		
		\noindent\textbf{Case 2}: If ${n\over(\log n)^{0.8}}\leq|I|\leq|S|$.
		
		Suppose for contradiction that  there exists a $I\subseteq S$ satisfies that the maximum matching $M$   in $\cup_{i\in I}\mathcal{H}_i$ is of size $|M|\leq mk(|I|-1)<mk|I|$.  We next to prove that w.h.p.  $M$ can be extended, which contradict the fact.
		
		Let $C_0=C_2\cap M$ and let $W'=W\cap V(M)$. We know that $C_2\setminus C_0\neq \phi,\ W\setminus W'\neq \phi$ since $|C_2|\gg|M|\gg|S|$. Let's consider the edges that incident with vertices of $W\setminus W'$. Since $M$ is the maximum matching,  for every $k$ element $e$ which consisted of a vertex $i\in I$  and $k-1$ vertices in $W\setminus W'$, there is
		\begin{align}
		\begin {cases}
		&e\notin E(H), or\\
		& e\in E(H)\ s.t.\begin{cases}
		&C(e)\in C_0, or\\
		&C(e)\notin C_2.\\
		\end{cases}\\
		\end{cases}
		\end{align}
		Otherwise, $\{(e\cup \{C(e)\})\setminus\{i\}\}\cup M$ will be a larger matching than  $M$. Let   $A$ denotes  event '(5)'. We have
		\begin{align*}
		\Pr[A]&\leq\left(1-p+p\left( {|C_0|+|C\backslash C_2|\over c}\right) \right)^{|I|{|W\setminus W'|\choose k-1}}=\left(1-p+p\left( {mk|I|\over(1+\varepsilon )n}+1-{\mu n\over (1+\varepsilon )n}\right) \right)^{|I|{|W\setminus W'|\choose k-1}}\\
		&=\left(1-p\left({\mu n-mk|I|\over(1+\varepsilon )n}\right) \right)^{|I|{|W\setminus W'|\choose k-1}}\leq\exp\left(-O_{k,m,\mu} p|I|{{n\over\log\log n}-m(k-1)j\choose k-1}\right)\\
		&\leq\exp\left(-O_{k,m,\mu}p|I|{{n\over\log\log n}\choose k-1}\right)\leq\exp\left(-O_{k,m,\mu}|I|{\log n\over(\log\log n)^{k-1}}\right).	
		\end{align*}
		Thus, taking the union we have
		\begin{align*}
		&{n\choose |W|}{|W|\choose |W'|}{c\choose|C_2|}{|C_2|\choose|C_0|}\sum_{|I|={n\over (\log n)^{0.8}}}^{{n\over (\log n)^{0.4}}}\exp\left(-O_{k,m,\mu}|I|{\log n\over(\log\log n)^{k-1}}\right)\\
		&\leq 16^n\sum_{|I|={n\over (\log n)^{0.8}}}^{{n\over (\log n)^{0.4}}}\exp\left(-O_{k,m,\mu}|I|{\log n\over(\log\log n)^{k-1}}\right)\\
		&\leq 16^n{n\over (\log n)^{0.4}}\exp\left(-O_{k,m,\mu} {n(\log n)^{0.2} \over(\log\log n)^{k-1}}\right)=o(1).	
		\end{align*}
	\end{proof}

	\section{Finding a colored rooted booster}\label{sec4}
	
	In the final step of the proof we shall need a Hamilton Berge path with two
	prescribed endpoints.  In the graph case this is usually achieved by joining
	a single auxiliary edge \(x'y'\), finding a Hamilton cycle through this edge,
	and then deleting it.  For Berge cycles in hypergraphs, we need extra require.  Indeed, if a \(k\)-edge \(e^*\) merely contains \(x'\)
	and \(y'\), a Berge cycle using \(e^*\) need not use \(x'\) and \(y'\) as the
	two core vertices corresponding to \(e^*\).  We therefore work throughout this
	section with a rooted virtual edge.
	
	\begin{definition}\label{def:rooted-virtual-edge}
		Let \(H_0\) be a \(k\)-uniform hypergraph on vertex set \(V\).  A rooted virtual
		edge is a triple
		\[\mathfrak e=(x,e^*,y),\]
		where \(x,y\in V\), \(x\ne y\), \(e^*\in\binom{V}{k}\setminus E(H_0)\), and
		\(\{x,y\}\subseteq e^*\).  The \(k\)-set \(e^*\) is not regarded as an edge of
		the original hypergraph and is assigned no colour.
	\end{definition}
	
	When we write \(H_0+\mathfrak e\), we mean the hypergraph obtained from \(H_0\)
	by joining the single edge \(e^*\), together with the additional
	requirement that \(e^*\), whenever it is used in a Berge path or cycle, is used
	with core pair \((x,y)\) or \((y,x)\).  Thus a Berge path 
	$v_0f_1v_1f_2\cdots f_\ell v_\ell$
	in \(H_0+\mathfrak e\) is said to contain \(\mathfrak e\) if \(f_i=e^*\) for
	some \(i\) and \(\{v_{i-1},v_i\}=\{x,y\}\).  A Berge cycle contains
	\(\mathfrak e\) with the analogous meaning.  We call such paths and cycles
	\(\mathfrak e\)-\emph{rooted}.
	
	\begin{definition}\label{def:rooted-booster}
		Let \(H_0\) be a \(k\)-uniform hypergraph on \(V\), and let
		\(\mathfrak e=(x,e^*,y)\) be a rooted virtual edge for \(H_0\).  A
		\(k\)-set
		\[
		f\in\binom{V}{k}\setminus (E(H_0)\cup\{e^*\})
		\]
		is called an \(\mathfrak e\)-booster for \(H_0\) if at least one of the
		following holds:
		{\rm (1)} \(H_0+\mathfrak e+f\) has fewer connected components than
		\(H_0+\mathfrak e\);
		{\rm (2)} \(H_0+\mathfrak e+f\) contains an \(\mathfrak e\)-rooted Berge path
		longer than every \(\mathfrak e\)-rooted Berge path in \(H_0+\mathfrak e\);
		{\rm (3)} \(H_0+\mathfrak e+f\) contains a Hamilton Berge cycle which contains
		\(\mathfrak e\).
	\end{definition}
	
	The rooting in Definition~\ref{def:rooted-booster} is essential.  Without it,
	the deletion of \(e^*\) from the final Hamilton Berge cycle would not
	necessarily leave a Berge path with endpoints \(x\) and \(y\).

	We use the standard P\'{o}sa ration-extension technique (see \cite{Posa1976}), but only rotations which preserve
	the rooted virtual edge are allowed.  Let
	\[
	P=v_0 f_1 v_1 f_2\cdots f_\ell v_\ell
	\]
	be a lengest \(\mathfrak e\)-rooted Berge path  in
	\(H_0+\mathfrak e\).  
	If there exists an edge  $e'\in E(H_0)\setminus E(P)$ satisfies $\{v_0, v_i\}\subset e'$  and $f_i\neq e$ for same $i\neq 0, l$, then  $P''=v_{i-1}e_{i-1}v_{i-2}\dots v_0e''v_i\dots e_lv_l$ be  another Berge path of length $l$ in $H_0+\mathfrak e$. We say that $P'$ is a \emph{admissible rotation} of  $P$ with \emph{ fixed endpoint $v_l$,  pivot $v_i$, new endpoint $v_{i-1}$} and \emph{broken edges $e'$}(see Figure \ref{Figure 1}).
	
	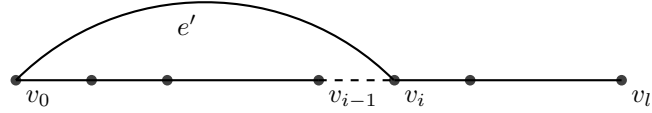
\begin{figure}[htp]
		\centering
		\begin{tikzpicture} [thick]%
		[spy using overlays={magnfication=5,size=1cm,connect}]
		\coordinate[label=-50:$v_0$] (v_0) at (0,0);
		\coordinate[label=-20:] (v_1) at (1,0);
		\coordinate[label=-20:] (v_2) at (2,0);
		\coordinate[label=-20:$v_{i-1}$] (v_{i-1}) at (4,0);
		\coordinate[label=-20:$e'$] (e') at (2,1);
		\coordinate[label=-20:$v_{i}$] (v_{i}) at (5,0);
		\coordinate[label=-20:] (v_{i+1}) at (6,0);
		\coordinate[label=-20:$v_l$] (v_{l}) at (8,0);
		\draw  (v_0)--(v_{i-1})   (v_{i})--(v_{l})  ;
		\draw [dashed] (v_{i-1})--(v_{i});
		\draw (v_0) arc (135:45:3.53);
		\foreach \p in {v_0,v_1,v_2,v_{i-1},v_{i},v_{i+1},v_{l}} \fill %
		[opacity=0.75] (\p)  circle   (2pt);
		\end{tikzpicture}
		\caption{$P'$ as a rotation of  $P$}\label{Figure 1}
	\end{figure}
	Let
	$
	\operatorname{End}(P,v_\ell)
	$
	be the set of all endpoints obtainable from \(v_0\) by a sequence of admissible
	rotations while keeping \(v_\ell\) fixed.  The set
	\(\operatorname{End}(P,v_0)\) is defined symmetrically.
	For a set \(X\subseteq V(P)\), write $$N_P(X):=\bigcup_{u\in X}(\{v\in f_{i}\mid u=v_i\}\cup\{v\in f_{i+1}\mid u=v_i\})$$ for the set of vertices of
	\(P\) which lie together with a vertex of \(X\) in one of the two path-edges
	incident with that vertex along \(P\). 	Since each vertex of \(X\) has at most two incident path-edges and each such
	edge contributes at most \(k-1\) further vertices, we have
	\[
	|N_P(X)|\le 2(k-1)|X|.
	\] 
	
	\begin{lemma}\label{lem:rooted-rotation-endpoints}
		Let \(H_0\) be a \(k\)-uniform hypergraph, let
		\(\mathfrak e=(x,e^*,y)\) be a rooted virtual dge for \(H_0\), and let
		$
		P=v_0 f_1 v_1\cdots f_\ell v_\ell
		$
		be a longest \(\mathfrak e\)-rooted Berge path in \(H_0+\mathfrak e\).  Put
		$
		U:=\operatorname{End}(P,v_\ell).
		$
		Then
		\[N_{H_0}(U)\setminus V(P) = \emptyset,\quad
		N_{H_0}(U)\setminus U
		\subseteq N_P(U)\cup e^* .
		\]
		Consequently,
		\[
		|N_{H_0}(U)\setminus U|
		\le 2(k-1)|U|+k .
		\]
	\end{lemma}
	
	\begin{proof}
		Fix \(u\in U\), and let \(P_u\) be an \(\mathfrak e\)-rooted Berge path with
		endpoints \(u\) and \(v_\ell\), obtained from \(P\) by admissible rotations.
		Consider an edge \(g\in E(H_0)\) containing \(u\).
		
		If \(g\notin E(P_u)\) and \(g\) contains a vertex outside \(V(P)\), then
		adjoining \(g\) at the endpoint \(u\) gives an \(\mathfrak e\)-rooted Berge
		path longer than \(P\), contradicting the maximality of \(P\).  Hence every
		vertex of \(g\) lies in \(V(P)\).
		
		If \(g\notin E(P_u)\) and \(g\) contains an internal vertex \(z\in V(P)\), then
		rotating \(P_u\) with pivot \(z\) and new edge \(g\) is admissible unless this
		would destroy the occurrence of the rooted edge \(\mathfrak e\) (see Figure \ref{Figure 2}).  In the
		exceptional case, the pivot lies in \(e^*\).  Thus \(z\in U\cup N_P(U)\cup e^*\).
		
		\begin{figure}[htp]
			\centering
			\begin{tikzpicture} [thick]%
			[spy using overlays={magnfication=5,size=1cm,connect}]
			\coordinate[label=-50:$v_i$] (v_i) at (0,0);
			\coordinate[label=-20:] (v_1) at (1,0);
			\coordinate[label=-20:] (v_2) at (2,0);
			\coordinate[label=-20:$z$] (z) at (4,0);
			\coordinate[label=-20:$e^*$] (e') at (4.3,0);
			\coordinate[label=-20:$g$] (e) at (2,1);
			\coordinate[label=-20:$v$] (v) at (5,0);
			\coordinate[label=-20:] (v_{i+1}) at (6,0);
			\coordinate[label=-20:] (v_{l-1}) at (8,0);
			\coordinate[label=-20:$v_l$] (v_l) at (9,0);
			\draw  (v_i)--(v_1)  (v_1)--(v_2)   (v_{i+1})--(v_{l-1}) (v_2)--(z)   (v)--(v_{i+1})   (v_{l-1})--(v_l)  ;
			\draw [dashed] (z)--(v);
			\draw (v_i) arc (135:45:3.53);
			\foreach \p in {v_i,v_1,v_2,z,v,v_{i+1},v_{l-1},v_l} \fill %
			[opacity=0.75] (\p)  circle   (2pt);
			\end{tikzpicture}
			\caption{$P_{v_i}$}\label{Figure 2}
		\end{figure}
		
		Finally suppose that \(g\in E(P_u)\).  If \(g\) was already an edge of the
		original path \(P\), then every vertex of \(g\) belongs to \(N_P(U)\cup U\).
		If \(g\) was introduced during an earlier rotation, let \(z\) be the pivot at
		the first step at which \(g\) entered the path.  The same argument applied at
		that first step shows that the relevant vertices of \(g\) lie in
		\(U\cup N_P(U)\cup e^*\).  Therefore
		\[
		N_{H_0}(U)\setminus U\subseteq N_P(U)\cup e^* .
		\]
		This proves the claimed bound.
	\end{proof}
	
	The preceding lemma motivates the following compact condition, which is the
	exact property needed in the booster argument.
	
	\begin{definition}\label{def:root-compatible}
		Let \(r\ge 1\).  A pair \((H_0,\mathfrak e)\), where
		\(\mathfrak e=(x,e^*,y)\) is a rooted virtual edge for \(H_0\), is called
		root-compatible at \(r\) if, for every longest \(\mathfrak e\)-rooted
		Berge path \(P\) in \(H_0+\mathfrak e\), both endpoint sets
		$
		\operatorname{End}(P,v_\ell)
		$ and $
		\operatorname{End}(P,v_0)
		$
		have size greater than \(r\).
	\end{definition}
	
	\begin{lemma}\label{lem-root-compatible}
		Let \(H_0\) be an \((r,3)\)-expander. For sufficiently large \(r\), the following holds: for every rooted
		virtual edge \(\mathfrak e\) for \(H_0\), the pair \((H_0,\mathfrak e)\)
		is root-compatible at  \(r\).
	\end{lemma}
	
	\begin{proof}
		Let \(P\) be a longest \(\mathfrak e\)-rooted Berge path, and put
		\(U:=\operatorname{End}(P,v_\ell)\).  Suppose for a contradiction that
		\(|U|\le r\).  By Lemma~\ref{lem:rooted-rotation-endpoints},
		\[
		|N_{H_0}(U)\setminus U|\le 2(k-1)|U|+k .
		\]
		On the other hand, since \(H_0\) is an \((r,4)\)-expander,
		\[
		|N_{H_0}(U)\setminus U|\ge 3(k-1)|U|.
		\]
		For \(n\) sufficiently large this is impossible, because
		$
		3(k-1)|U|>2(k-1)|U|+k
		$
		for every nonempty \(U\) and every \(k\ge 3\).  Thus \(|U|>r\).  The proof for
		\(\operatorname{End}(P,v_0)\) is identical.
	\end{proof}

	In the final construction the auxiliary hypergraph may contain a sparse set of	vertices of degree two, coming from the special small vertices.  The same	conclusion as Lemma~\ref{lem-root-compatible} remains valid
	if these degree-two vertices  pairwise separated by Berge distance at least five, and if the rooted
	virtual edge is not incident with any of them. 
	\begin{corollary}\label{corollary}
		Let \(H_0\) be an \((r,2)\)-expander. For sufficiently large \(r\), the following holds: for every rooted
		virtual edge \(\mathfrak e=(x,e^*,y)\) for \(H_0\), if neither $x$ nor $y$ is the vertex of degree two, then  the pair \((H_0,\mathfrak e)\)
		is root-compatible at  \(r\).	
	\end{corollary}

	\begin{lemma}\label{lem:many-rooted-boosters}
		Let \(k\ge 3\) be fixed.  Let \(H_0\) be a \(k\)-uniform hypergraph on  vertex
		set \(V\) of size \(m\), and let \(\mathfrak e=(x,e^*,y)\) be a rooted virtual
		edge for \(H_0\).  Assume that \((H_0,\mathfrak e)\) is root-compatible at	 \(r\), where \(r\) is large enough, that every connected component of
		\(H_0+\mathfrak e\) has more than \(r\) vertices, and that
		$|E(H_0)|=O(m).$
		Then either \(H_0+\mathfrak e\) contains a Hamilton Berge cycle containing
		\(\mathfrak e\), or the number of \(\mathfrak e\)-boosters for \(H_0\) is at least	$\gamma_k r^2 m^{k-2},$	where \(\gamma_k>0\) depends only on \(k\).
	\end{lemma}
	
	\begin{proof}
		If \(H_0+\mathfrak e\) is disconnected, then by assumption each connected
		component has more than \(r\) vertices.  Hence there are at least \(r^2\)
		unordered pairs of vertices lying in distinct components.  Every \(k\)-set
		containing such a pair and  contained in \(\binom{V}{k}\) decreases the number
		of connected components, except for the \(O(m)\) already present edges of
		\(H_0\) and the single set \(e^*\).  Since \(r\to\infty\) and \(m\to\infty\),
		this gives at least \(\gamma_k r^2m^{k-2}\) boosters after decreasing
		\(\gamma_k\), if necessary.
		
		We may therefore assume that \(H_0+\mathfrak e\) is connected and contains no
		Hamilton Berge cycle through \(\mathfrak e\).  Let
		$
		P=v_0 f_1 v_1\cdots f_\ell v_\ell
		$
		be a longest \(\mathfrak e\)-rooted Berge path.  By Lemma \ref{lem-root-compatible},
		\[
		X:=\operatorname{End}(P,v_\ell)
		\quad\text{satisfies}\quad |X|>r .
		\]
		For every \(a\in X\), choose an \(\mathfrak e\)-rooted longest Berge path \(P_a\)
		with endpoints \(a\) and \(v_\ell\).  Applying Lemma \ref{lem-root-compatible} to \(P_a\)
		while keeping endpoint \(a\) fixed gives a set \(\operatorname{End}(P_a,a)\) of more than \(r\) possible endpoints.  For every \(b\in \operatorname{End}(P_a,a)\), there is an
		\(\mathfrak e\)-rooted Berge path \(P_{a,b}\) of length \(\ell\) with endpoints
		\(a\) and \(b\).
		
		Since \(H_0+\mathfrak e\) has no Hamilton Berge cycle through \(\mathfrak e\),
		the path \(P_{a,b}\) is not Hamilton Berge path, or the pair \(a,b\) is not  closed
		by an unused edge.  In either case, any new \(k\)-set \(f\) containing
		\(\{a,b\}\) and not belonging to \(E(H_0)\cup\{e^*\}\) is an
		\(\mathfrak e\)-booster: it either closes \(P_{a,b}\) into a Hamilton Berge
		cycle through \(\mathfrak e\), or it produces a strictly longer
		\(\mathfrak e\)-rooted Berge path by extending the cycle along a vertex not on
		\(P_{a,b}\).
		There are at least \(r(r-1)\) ordered pairs \((a,b)\) arising in this way.  For
		each such ordered pair, the number of \(k\)-sets in \(\binom{V}{k}\) containing
		\(\{a,b\}\) is \(\binom{m-2}{k-2}\).  A fixed \(k\)-set contains at most
		\(k(k-1)\) ordered pairs of vertices.  Removing the \(O(m)\) existing edges of
		\(H_0\) and the virtual edge \(e^*\), we obtain
		\[
		\#\{\mathfrak e\text{-boosters}\}
		\ge
		\frac{1}{k(k-1)}
		\left(r(r-1)\binom{m-2}{k-2}-O(m)k(k-1)\right).
		\]
		Because \(r\to\infty\) and \(m\to\infty\), the desired bound follows with a constant \(\gamma_k>0\).
	\end{proof}

	The next statement is the form used in the iterative part of the final proof.
	It is deliberately stated uniformly over all sparse auxiliary hypergraphs,
	rooted virtual edges, and large colour sets.
	
	\begin{theorem}\label{thm:available-rooted-boosters}
		Let \(k\ge 3\) be fixed, and let \(\varepsilon,\beta,K>0\).  Let
		$
		c\ge (1+\varepsilon)n,
		$ and 
		$\frac{(k-1)!\log n}{n^{k-1}}
		\le p\le
		\frac{2(k-1)!\log n}{n^{k-1}}.
		$
		Then \(H\sim H_c^{(k)}(n,p)\) has the following property w.h.p.
		Let \(H_0\subseteq H\) be any rainbow \(k\)-uniform hypergraph on a vertex set
		\(V_0\subseteq[n]\) such that
		\[
		\frac{n}{\log\log n}\le m:=|V_0|\le \frac{2n}{\log\log n},
		\qquad
		|E(H_0)|\le Km,
		\]
		and let \(\mathfrak e=(x,e^*,y)\) be any rooted virtual edge for \(H_0\).
		Assume that 
		
		{\rm(1)} \((H_0,\mathfrak e)\) is root-compatible at  \(\beta m\), that
		every connected component of \(H_0+\mathfrak e\) has more than \(\beta m\)
		vertices, and 
		
		{\rm(2) } \(H_0+\mathfrak e\) contains no Hamilton Berge cycle
		containing \(\mathfrak e\), and 
		
		{\rm (3)}   \(A\subseteq[c]\setminus C(H_0)\) satisfies
		$
		|A|\ge \varepsilon n/10.
		$
		
		\noindent	Then there exists an edge
		$
		f\in E(H)\setminus E(H_0)
		$
		such that \(f\) is an \(\mathfrak e\)-booster for \(H_0\) and
		\(C(f)\in A\).
	\end{theorem}
	
	\begin{proof}
		Fix \(V_0,H_0,\mathfrak e\), and \(A\) satisfying the hypotheses.  By
		Lemma~\ref{lem:many-rooted-boosters}, the number of candidate
		\(\mathfrak e\)-boosters contained in \(\binom{V_0}{k}\) is at least
		\[
		\gamma_k(\beta m)^2m^{k-2}
		=\gamma_k\beta^2m^k .
		\]
		Each such \(k\)-set appears in \(H\) and receives a colour from \(A\) with
		probability
		$
		p\frac{|A|}{c}
		\ge
		\frac{\varepsilon}{20}p,
		$
		for all sufficiently large \(n\).  Therefore the probability that none of
		these boosters appears with a colour in \(A\) is at most
		\[(1-\frac{\varepsilon}{20}p)^{\gamma_k\beta^2m^k}\le
		\exp\left(-\frac{\varepsilon\gamma_k\beta^2}{20}pm^k\right).
		\]
		Since \(m\ge n/\log\log n\), the exponent is at least
		$
		\Omega\left(\frac{n\log n}{(\log\log n)^k}\right).
		$
		
		It remains to take a union bound.  The number of choices for \(V_0\) is at most
		$
		\sum_{m\le 2n/\log\log n}\binom{n}{m}
		\le 2^n.
		$
		For a fixed \(V_0\), the number of choices for a hypergraph \(H_0\) with at
		most \(Km\) edges is at most
		\[
		\sum_{j\le Km}
		\binom{\binom{m}{k}}{j} p^j
		\le
		\left(\frac{e\binom{m}{k}p}{Km}\right)^{Km}
		\le
		\exp\left(O\left(m\log\log n\right)\right).
		\]
		The number of possible rooted virtual edges is at most \(m^k k^2\), and
		the number of possible colour sets \(A\) is at most \(2^c=\exp(O(n))\).

		Consequently the total failure probability is bounded by
		\[
		\sum_{m=n/\log\log n}^{2n/\log\log n}
		2^n
		2^c m^{k+2}
		\exp\left(O(m\log\log n)\right)
		\exp\left(-\Omega\left(\frac{n\log n}{(\log\log n)^k}\right)\right)
		\]
		\[
		\le {2n\over \log\log n}\exp(O(n))\cdot	
		\exp\left(-\Omega\left(\frac{n\log n}{(\log\log n)^k}\right)\right)
		\]
		for every fixed \(k\).  Hence the above sum is \(o(1)\) for sufficiently large \(n\), completing the proof.
	\end{proof}
	
	\begin{remark}[Use in the iterative construction]\label{rem:iterative-use}
		In the final proof, Theorem~\ref{thm:available-rooted-boosters} is applied
		repeatedly.  At step \(i\), the hypergraph \(H_i\) is the current rainbow
		auxiliary hypergraph and \(A_i\) is the set of colours not yet used by
		\[
		P,\quad H_i,\quad e_x,\quad e_y .
		\]
		The theorem supplies a genuine edge of \(H\) with colour in \(A_i\).  This edge
		is then added to \(H_i\), while the rooted virtual edge \(\mathfrak e\) remains uncoloured throughout.
	\end{remark}

	\section{Proof of Theorem 1.1.}

	\begin{proof}[Proof of  Theorem \ref{Theorem 1.1}]
		Let $\eta>0$ be a sufficiently small constant. We assume that $H$ satisfies all  results from the previous section.
		
		For each vertex $v\in SMALL$, we arbitrarily choose  two distinct edges $e_v^1$ and $e_v^2$   of incident to $v$ in $H$. 
		Put\[
		E_{S}:=\{e_v^1,e_v^2:v\in SMALL\}, \qquad
		V_{S}:=\bigcup_{e\in E_{S}} e.
		\]
		By (4) of Lemma \ref{Lemma 2.2}, we have that $\{e_v^1, e_v^2\}\cap \{e_u^1, e_u^2\}=\phi$ for any  vertext $v,u\in SMALL$, and by (3) of Lemma \ref{ Lemma 2.5} we have that $E_S$ is rainbow. Let $C_S=\{c(e): e\in E_S\}$ denote the set of colors used in $E_S$. Let $C'=C\backslash C_S$ denote the set of colors unused. It is easy to see $|C'|=c-2|SMALL|\geq (1+{\varepsilon\over 2}n)$ when $\eta$ is small enough.

		Let	$
		H_{\mathrm{clean}}
		:=H\bigl[[n]\setminus V_{ S}\bigr]
		\setminus \{e:\chi(e)\in C_{S}\}.
		$
		Then every vertex
		$v\in [n]\setminus V_{S}$, we have	
		$
		d_{H_{\mathrm{clean}}}(v)\ge \frac{\eta}{2}\log n .
		$
		Indeed,	 by  \(|SMALL|\le n^{0.4}\), and consequently
		$|V_{S}|\le 2k n^{0.4}.
		$
		The expected number of edges incident with a fixed vertex \(v\notin
		V_{S}\) and meeting \(V_{S}\) is
		\[
		O\left(|V_{S}|\, n^{k-2}p\right)
		=O(n^{-0.6}\log n)=o(1).
		\]
		It remains to control the deletion of colours from \(C_{S}\).    Since
		\[
		|C_{S}|\le 2|E_{S}|\le 2n^{0.4}
		\quad\text{and}\quad c\ge(1+\varepsilon)n,
		\]
		the expected number of incident edges at \(v\) whose colour lies in
		\(C_{S}\) is \(O(n^{-0.6}d_H(v))=o(1)\).  A union bound over all 
		vertices, together with the  bound \(d_H(v)=O(\log n)\) for all vertices
		outside a probability \(o(1)\) event, shows that no  vertex loses more
		than \(o(\log n)\) edges through the colour deletion.
		Finally, every vertex under consideration is outside \(SMALL\), and hence
		$
		d_H(v)>\eta\log n.
		$
		Combining the preceding estimates gives
		\[
		d_{H_{\mathrm{clean}}}(v)
		\ge \eta\log n-o(\log n)
		\ge \frac{\eta}{2}\log n,
		\]
		as required.

		We now describe the final assembly in a way that keeps the real hypergraph, the
		auxiliary matching, the colour sets, and the single virtual edge separate.
		Let $\mu={\varepsilon\over 10}$,  by applying Lemma \ref{Lemma 3.7.} to $H[[n]\backslash V_S;C']$, we can find a subsets $W\subseteq ([n]\backslash V_S)$ and two disjoint sets $C_1, C_2\subseteq C'$ satisfies the following results$:$
		
		(1) $|W|=(1+o(1)){n\over \log\log n}$, and
		
		(2) $|C_1|=(1+o(1))\mu n, |C_2|=(1+o(1))\mu n$, and
		
		(3) for every $v\in SMALL$, there is $d^{C_2}(v,{W\choose k-1})\in ({\mu\eta\log n\over 2(\log \log n)^{k-1}}, {2\mu\log n\over (\log\log n)^{k-1}})$, and
		
		(4) the subhypergraph $H[W,C_1]$ satisfies all the conditions of Theorem \ref{Theorem 3.9.}
		
		\noindent Next by applying Theorem \ref{Theorem 3.9.} to $H[W,C_1]$, we can find a subhypergraph $H'_1\subseteq H[W,C_1]$ such that
		
		(a)	$H'_1$ is rainbow, and
		
		(b)	$H'_1$ is an $({\mu\eta|W|\over 100},100)$-expander, and
		
		(c)	$e(H'_1)=O({n\over \log\log n})$.

		By applying Theorem \ref{ Theorem 2.4} and  Lemma \ref{ Lemma 2.5} to $H[V\backslash(V_S\cup W); C'\backslash (C_1\cup C_2)]$, we find a rainbow Berge path $P$.
		Second, let $x,y$ denote the endpoints of path $P$, and define $S'=[n]\backslash(SMALL\cup V(P)\cup W)$, $S=S'\cup \{x, y\}$.
		We truncate
		\(P\), if necessary, so that the set
		\[
		S_0:=[n]\setminus\bigl(V(P)\cup W\cup SMALL\bigr)
		\]
		satisfies
		$
		\frac{n}{(\log n)^{0.4}}
		\le |S_0|
		\le \frac{2n}{(\log n)^{0.4}} .
		$ 
		The above choice is made precisely so that \ref{Theorem 3.12.} is
		applied inside its stated range.

		Let \(x\) and \(y\) be the two endpoints of \(P\).  Define
		\[
		S:=S_0\cup\{x,y\}.
		\]Consider the edge-colored $k$-uniform bipartite hypergraph $H_2:=H[S, {W\choose k-1}; C_2]$. 
		By Theorem \ref{Theorem 3.12.}, there is a rainbow \(10\)-matching \(M\) from \(S\) into
		the set \(W\), using colours disjoint from \(C(P)\cup C_{S}\cup C_1\).  
		Choose one edge of \(M\) incident with \(x\), denoted by \(e_x\), and one edge
		of \(M\) incident with \(y\), denoted by \(e_y\).  Choose vertices
		\[
		x'\in e_x\cap W,\qquad y'\in e_y\cap W .
		\]
		Note that $C(e^x)\neq C(e^y)$ since $M$ is rainbow, and $\{e_x, e_y\}\cup P$ is a rainbow  Berge path.

		The matching used inside the auxiliary graph is
		\[
		M^\circ
		:=M\setminus \{e\in M:e\cap\{x,y\}\ne\varnothing\}.
		\]
		Thus every edge of \(M^\circ\) is disjoint from \(\{x,y\}\).  This is essential:
		the endpoints \(x,y\) belong to the long path \(P\), and the expander to which
		we apply the booster argument is built only on the vertices outside \(P\).
		We define the real auxiliary hypergraph
		\[
		G_0
		:=
		\bigl(
		[n]\setminus V(P),
		\,
		E_{S}\cup E(H_1')\cup E(M^\circ)
		\bigr).
		\]
		All edges of \(G_0\) are genuine edges of \(H\), and the colour sets
		$
		C_{SMALL}$, $C(R)$ and $C(M^\circ)$
		are pairwise disjoint.  Hence \(G_0\) is rainbow.
		
		Choose a rooted virtual edge \(\mathfrak e=e(x',e^*,y')\) for $H$ such that the \(k\)-set
		\(
		e^*\subseteq W\) with no colour.
		Such a choice is possible because \(|W|\to\infty\), whereas \(H[W]\) contains
		only \(o(\binom{|W|}{k})\) edges w.h.p.   	Since $H'_1$ is  $({\mu\eta|W|\over 100}, 100)$-expander, and there exists a rainbow $10$-matching from $S'$ to $W$. Thus,  Lemma \ref{Lemma 3.5.} follows  that $G_0$ is an  $({\mu\eta|W|\over 100}, 4)$-expander. Similarly, since each Berge path with endpoints all in $SMALL$ is of length least $5$  by (5) of Lemma \ref{Lemma 2.2}. Thus, Lemma \ref{Lemma 3.2} follows that $G_0$ be an $({\mu\eta|V_0|\over 200}, 2)$-expander.
		
		For \(i\ge 0\), suppose that \(G_i\subseteq H\) is a rainbow hypergraph on
		\([n]\setminus V(P)\), containing \(G_0\), and using no colour from
		$
		C(P)\cup\{\chi(e_x),\chi(e_y)\}.
		$
		Set
		$
		\widehat G_i:=G_i\cup\{e^*\}.
		$
		The Theorem \ref{thm:available-rooted-boosters} is applied to \(\widehat G_i\), with the rooted virtual edge 
		as \((x',e^*,y')\), but every booster added to \(G_i\) is a genuine edge of
		\(H\) and has a colour in
		\[
		A_i:=[c]\setminus
		\bigl(C(P)\cup C(G_i)\cup\{c(e_x),c(e_y)\}\bigr).
		\]
		Since
		$
		|[n]\setminus V(P)|=O\!\left(\frac{n}{(\log n)^{0.4}}\right),
		$
		the number of booster steps is \(O(n/(\log n)^{0.4})=o(n)\).  Consequently,
		at every step
		$
		|A_i|\ge \frac{\varepsilon}{2}n
		$
		for all sufficiently large \(n\).  By  Corollary \ref{corollary} and  Theorem \ref{thm:available-rooted-boosters}, unless \(\widehat G_i\) already contains a Hamilton Berge
		cycle using the rooted virtual edge \((x',e^*,y')\), there is an available
		genuine rainbow booster which may be added to \(G_i\).
		
		After at most \(|[n]\setminus V(P)|\) iterations, we obtain a rainbow
		hypergraph \(G_t\subseteq H\) such that
		$
		\widehat G_t=G_t\cup\{e^*\}
		$
		contains a Hamilton Berge cycle \(C\) on the vertex set
		\([n]\setminus V(P)\), and the cycle uses the virtual edge \(e^*\) with core
		endpoints \(x'\) and \(y'\).  Deleting \(e^*\) from \(C\) yields a rainbow
		Berge path \(Q\) in the genuine hypergraph \(G_t\), whose endpoints are
		\(x'\) and \(y'\), and whose vertex set is \([n]\setminus V(P)\).
		
		It remains only to concatenate the three genuine pieces.  Orient \(P\) from
		\(x\) to \(y\), and orient \(Q\) from \(y'\) to \(x'\).  Then
		$
		P\cup e_y\cup Q\cup e_x
		$
		form a Hamilton Berge cycle of \(H\).  The colours are distinct because
		\[
		C(P),\quad C(Q),\quad \{C(e_x),C(e_y)\}
		\]
		are pairwise disjoint by construction.  The vertices are all covered exactly
		once as core vertices: \(P\) covers \(V(P)\), while \(Q\) covers
		\([n]\setminus V(P)\).  Therefore \(H\) contains a rainbow Hamilton Berge
		cycle.
		
	\end{proof}

	\section{Conclusion}
	
	In this paper, we have fully resolved the problem of rainbow Hamilton Berge cycles in edge-colored random $k$-uniform hypergraphs for all $k \ge 3$. We have proved that if the number of colors $c = (1+o(1))n$ and the edge probability $p = (k-1)!(\log n + \log\log n + \omega(n))/n^{k-1}$, then the random edge-colored $k$-uniform hypergraph $H_c^k(n,p)$ contains a rainbow Hamilton Berge cycle with high probability. Moreover, both thresholds are asymptotically tight.
	
	Our work extends the celebrated result of Ferber and Krivelevich (2016) from graphs to hypergraphs, establishing a unified threshold framework across all uniformities. The proof adapts and extends the three-phase absorption method to the hypergraph setting, overcoming the substantial obstacles introduced by higher-order edges and the global rainbow constraint. The main technical contribution lies in the development of a \emph{rainbow hypergraph absorption method}, which skillfully combines tools such as the Lov{\'a}sz Local Lemma, the Aharoni--Haxell hypergraph matching theorem, and a P{\'o}sa-type rotation-extension argument adapted to Berge paths, allowing us to manage the intricate dependencies in sparse random hypergraphs while preserving color-distinctness throughout the construction.
	
	Beyond its theoretical contribution, our result has meaningful implications for the design and analysis of real-world networks involving group interactions, such as multi-terminal communication systems, distributed storage networks, and biochemical reaction networks. The established thresholds provide explicit conditions under which a conflict-free, fully covering cyclic structure can be guaranteed in randomly colored random hypergraphs.

	Several interesting questions remain open:
	\begin{itemize}
		\item \textbf{Sharp thresholds:} Can the $o(1)$ terms in the parameters $c$ and $p$ be made more precise? In particular, what is the exact limiting probability when $p = (k-1)!\frac{\log n + \log\log n + c}{n^{k-1}}$ for constant $c$?
		\item \textbf{Resilience:} How robust is rainbow Hamiltonicity in random hypergraphs? For graphs, there are known results on the local and global resilience of Hamiltonicity~\cite{BenShimonKrivelevichSudakov2011}; extending these to the hypergraph setting would be natural.
		\item \textbf{Other cycle types:} Our work focuses on Berge cycles. Similar questions for other types of cycles in hypergraphs (such as loose cycles~\cite{DudekEnglishFrieze2018}, \(\ell\)-cycles~\cite{HanHanZhao2022} remain to be investigated in the rainbow setting.
		\item \textbf{Algorithms:} Can our probabilistic proof be converted into an efficient randomized algorithm for finding rainbow Hamilton Berge cycles in random hypergraphs?
		\item \textbf{Universality:} Ferber et al.~\cite{FerberNenadovPeter2016,Condon, Montgomery
		} studied universal properties in random graphs. It would be interesting to investigate whether random edge-colored hypergraphs are rainbow universal for certain classes of hypergraphs.
	\end{itemize}
	
	\textbf{Concluding remarks}
	This work contributes to the growing body of research on rainbow structures in random discrete objects. By bridging the gap between the well-understood graph case and the more complex hypergraph setting, we provide new insights into how combinatorial properties interact with coloring constraints in higher uniformities. The techniques developed here may find applications in other problems involving rainbow subgraphs in random hypergraphs and in the study of resource allocation problems modeled by edge-colored hypergraphs.


\begin{thebibliography}{99}
		\bibitem{Berge1970}
		Berge C. (1970) Graphes et Hypergraphes. https://api.semanticscholar.org/CorpusID:119017740.
		
		\bibitem{Karp1972}
		Karp R M. (1972) Reducibility among combinatorial problems.
		In: Complexity of Computer Computations, New York: Plenum, pp.~ 85--103.
		
		
		\bibitem{FriezeLoh2014} 
		Frieze A,  Loh P S. (2014)
		Rainbow Hamilton cycles in random graphs.
		{\it Random Structures \& Algorithms} \textbf{44}(3), 328--354.
		
		\bibitem{FerberKrivelevich2016}
		Ferber A, Krivelevich M. (2016)
		Rainbow Hamilton cycles in random graphs and hypergraphs.
		{\it Recent trends in combinatorics}  \textbf{159}, 167--189.
		
		\bibitem{BalBerkowitzDevlinSchacht2021}
		Bal D,  Berkowitz R, Devlin P, Schacht M. (2021)
		Hamiltonian Berge cycles in random hypergraphs.
		{\it Combinatorics, Probability and Computing \/}  \textbf{30}(2), 228--238.
		
		\bibitem{DudekEnglishFrieze2018}
		Dudek A,  English S, Frieze A. (2018)
		On Rainbow Hamilton Cycles in Random Hypergraphs.
		{\it The Electronic Journal of Combinatorics} \textbf{25}(2), 55--68.
		
		
		
		\bibitem{AharoniHaxell2000}
		Aharoni R, Haxell P. (2000)
		Hall's theorem for hypergraphs. 
		{\it Journal of Graph Theory} \textbf{35}(2), 83--88.
		
		
		
		\bibitem{BenEliezerKrivelevichSudakov2012}
		Ben-Eliezer I,  Krivelevich M, Sudakov B. (2012)
		The size Ramsey number of a directed path.
		{\it	Journal of Combinatorial Theory Series B} \textbf{102}, 743--755.
		
		\bibitem{BenShimonFerberHefetzKrivelevich2010}
		Ben-Shimon S, Ferber A,  Hefetz D, Krivelevich M. (2010)
		Hitting time results for Maker-Breaker games.
		{\it	Random Structures \& Algorithms} \textbf{41}(1), 23--46.
		
		\bibitem{AlonSpencer2000}
		Alon N,  Spencer J. (2000) 
		{\it The probabilistic method} (2nd edition). 
		New York: Wiley Press. 
		
		\bibitem{FioriniWilson1977}
		Fiorini S,  Wilson R J. (1977)
		{\it	Edge-colourings of graphs}. 
		Pitman,  ISBN 0273011294.
		
		\bibitem{Haxell1995}
		Haxell P E. A condition for matchability in hypergraphs. (1995)
		{\it	Graphs and Combinatorics} \textbf{11}(3), 245--248.
		
		
		\bibitem{Posa1976}
		P\'osa L. Hamiltonian circuits in random graphs. (1976)
		{\it	Discrete Mathematics} \textbf{14}(4), 359--364.
		
		\bibitem{BenShimonKrivelevichSudakov2011}
		Ben-Shimon S, Krivelevich M, Sudakov B. (2011)
		Local Resilience and Hamiltonicity MakerBreaker Games in Random Regular Graphs.
		{\it	Combinatorics Probability \& Computing} \textbf{20}(2), 173--211.
		
		
		\bibitem{HanHanZhao2022}
		Han H, Han J, Zhao Y. (2022)
		Minimum degree thresholds for Hamilton \(k/2\)-cycles in \(k\)-uniform hypergraphs.
		{\it	Journal of Combinatorial Theory, Series B} \textbf{153}, 105--148.
		
		\bibitem{FerberNenadovPeter2016}
		Ferber A,  Nenadov R, Peter U. (2016)
		Universality of random graphs and rainbow embedding.
		{\it	Random Structures \& Algorithms} \textbf{48}(3), 546--564.
		
		\bibitem{Condon} 
		Condon P,  Espuny D{\'\i}az A,  Gir{\~a}o A,  K{\"u}hn D, and  Osthus D.  (2021)  Dirac's theorem for random regular graphs.
		{\it Combinatorics, Probability and Computing} \textbf{30}(1), 17--36.
		
		\bibitem{Montgomery}
		Montgomery R. (2020) 
		Hamiltonicity in random directed graphs is born resilient. 
		{\it Combinatorics, Probability and Computing} \textbf{29}(6), 900--942.	
	\end{thebibliography}
\end{document}